\documentclass[11pt,reqno]{amsart}

\usepackage[T1]{fontenc}
\usepackage{enumitem}
\usepackage{hyperref}
\hypersetup{
  colorlinks   = true,
  citecolor    = blue,
  linkcolor    = blue
}
\usepackage{amsmath,amsfonts,amsthm,amssymb,color,tikz, comment,csquotes}
\usepackage{mathrsfs}   
\usepackage{pdfsync}
\usepackage[font={scriptsize}]{caption}

\numberwithin{equation}{section}

\usepackage[left=1in, right=1in, top=1.1in,bottom=1.1in]{geometry}
\newcommand{\dw}{\dot{W}}

\newcommand{\ii}{\imath}

\newcommand{\1}{{\bf 1}}

\newcommand{\xiti}{\tilde{\xi}}
\newcommand{\lati}{\tilde{\la}}

\newcommand{\kti}{\tilde{K}}
\newcommand{\psiti}{\tilde{\psi}}

\newcommand{\bw}{\mathbf{W}}

\newcommand{\bchi}{\mathbf{\chi}}
\newcommand{\uchi}{\pmb{\chi}}
\newcommand{\uw}{\pmb{W}}

\newcommand{\scal}{\mathfrak{s}}

\newcommand{\tti}{\tilde{t}}
\newcommand{\yti}{\tilde{y}}

\newcommand{\lastchange}[1]{{#1}}

\newcommand{\E}{\mathbb E}
\newcommand{\R}{\mathbb R}

\newcommand{\PP}{\mathbb P}

\newcommand{\Z}{\mathbb Z}
  
\newcommand{\be}{\mathbf{E}}

\newcommand{\bh}{\mathbf{H}}

\newcommand{\cb}{\mathcal B}
\newcommand{\cac}{\mathcal C}

\newcommand{\cd}{\mathcal D}
\newcommand{\ce}{\mathcal E}
\newcommand{\cf}{\mathcal F}

\newcommand{\ci}{\mathcal I}
\newcommand{\cj}{\mathcal J}

\newcommand{\cn}{\mathcal N}

\newcommand{\cs}{\mathcal S}

\newcommand{\cfs}{\mathcal{F}^{\textsc s}}
    
\newcommand{\al}{\alpha}

\newcommand{\ga}{\gamma}

\newcommand{\ka}{\kappa}
\newcommand{\la}{\lambda}

\newcommand{\oom}{\Omega}

\newcommand{\si}{\sigma}

\newcommand{\vp}{\varphi}

\newcommand{\lp}{\left(}
\newcommand{\rp}{\right)}

\newtheorem{theorem}{Theorem}[section]

\newtheorem{definition}[theorem]{Definition}

\newtheorem{lemma}[theorem]{Lemma}

\newtheorem{proposition}[theorem]{Proposition}

\theoremstyle{remark}
\newtheorem{remark}[theorem]{Remark}
\theoremstyle{definition}

\newcommand{\bean}{\begin{eqnarray*}}
\newcommand{\eean}{\end{eqnarray*}}
\newcommand{\ben}{\begin{enumerate}}
\newcommand{\een}{\end{enumerate}}
\newcommand{\beq}{\begin{equation}}
\newcommand{\eeq}{\end{equation}}

\begin{document}

\title{A $K$-rough path above the space-time fractional Brownian motion}

\author[X. Chen \and A.  Deya \and C. Ouyang \and S. Tindel]
{Xia Chen \and Aur{\'e}lien  Deya \and Cheng Ouyang \and Samy Tindel}

\address{Xia Chen:  Department of Mathematics, University of Tennessee Knoxville,  TN 37996-1300, United States.}
\email{xchen@math.utk.edu}

\address{Aur\'elien Deya: Institut Elie Cartan, University of Lorraine
B.P. 239, 54506 Vandoeuvre-l\`es-Nancy, Cedex
France.}
\email{Aurelien.Deya@univ-lorraine.fr}

\address{Cheng Ouyang: Department of Mathematics, 
Statistics and Computer Science, University of Illinois at Chicago, Chicago, 
United States.}
\email{couyang@math.uic.edu}
\thanks{C. Ouyang is supported in part by Simons grant \#355480}

\address{Samy Tindel: Department of Mathematics, Purdue University, 150 N. University Street, West Lafayette, IN 47907, United States.}
\email{stindel@purdue.edu}
\thanks{S. Tindel is supported by the NSF grant  DMS-1952966.}



\begin{abstract}
We construct a $K$-rough path (along the terminology of \cite[Definition 2.3]{De16}) above either a space-time or a spatial fractional Brownian motion, in any space dimension $d$. This allows us to provide an interpretation and a unique solution for the corresponding parabolic Anderson model, understood in the renormalized sense. We also consider the case of a spatial fractional noise.
\end{abstract}

\maketitle

\

\section{Introduction}\label{sec:intro}

The main objective of the analysis in this paper is to provide a wellposedness statement for the following parabolic Anderson model:

\begin{equation}\label{eq:she-intro}
\begin{cases}
\partial_{t} u_{t}(x) = \frac12 \Delta u_{t}(x) + u_{t}(x) \, \dot W_{t}(x), &\quad t\in\R_{+},\ x\in \R^{d},\\
u_0=\Psi
\end{cases}
\end{equation}
\textit{in situations where $\dot W$ corresponds to a space-time fractional noise of low regularity}.

\smallskip

Formally, the covariance function of such a noise $\dot{W}$ can be written as
\begin{equation}\label{eq:def-gamma}
\be\big[  \dw_{t}(x) \, \dw_{s}(y) \big]
=
\ga_{0}(t-s) \, \ga(x-y),
\end{equation}
with $\ga_0$ and $\ga$ the distributions, given in Fourier modes by
\begin{equation}\label{eq:rep-gamma-fourier}
\ga_{0}(t) =  c_{0}\int_{\R} e^{\ii \la t} |\la|^{1-2H_{0}}d\lambda
\quad\text{and}\quad
\ga(x) =  c_{\bh}\int_{\R^{d}} e^{\ii \xi \cdot x} \prod_{j=1}^{d} |\xi_{j}|^{1-2H_{j}}d\xi,
\end{equation}
where $\bh$ denotes the vector $(H_1,\ldots,H_d)$ and where $c_{H_0},c_{\bh}$ are the positive constants explicitly given by 
\begin{equation}\label{cstt-c-h}
c_{H_0}=\bigg(\int_{\R} d\xi \, \frac{|e^{\imath \xi}-1|^2}{|\xi|^{2H_i+1}}\bigg)^{-1/2} \,  , \quad c_{\bh}=\bigg(\prod_{i=1}^d \int_{\R} d\xi \, \frac{|e^{\imath \xi}-1|^2}{|\xi|^{2H_i+1}}\bigg)^{-1/2}.
\end{equation}

\smallskip

At this point, it should already be noted that a Skorohod interpretation and treatment of the model in the rough environment \eqref{eq:def-gamma} has recently been carried out by one of the authors in \cite{Ch18}, using a delicate analysis of intersection local times. We have then extended these considerations in~\cite{CDOT}, and therein provided sharp moment estimates on the Skorohod solution. 

\smallskip

In contrast with the latter investigations, we here would like to study equation \eqref{eq:she-intro} along a Stratonovich (or pathwise) interpretation. The basic idea behind this approach can be roughly expressed in terms of approximation procedures. Namely, we first introduce a sequence $\{\dot{W}^n; n\geq1\}$ of smooth approximations of $\dot{W}$, which can for instance be given by a mollyfing procedure
\begin{align}\label{def:approximate noise-intro}
\dw^n:=\partial_t  \partial_{x_1} \cdots \partial_{x_d} W^n, \quad \text{where} \ W^n:=\rho_n \ast W \ \text{and} \ \rho_n(s,x):=2^{n(d+2)}\rho(2^{2n} s,2^n x),
\end{align}
for some mollifier $\rho: \R^{d+1} \to \R_+$ satisfying standard regularity assumptions. Then consider the sequence $\{u^n; n\geq1\}$ of classical solutions associated with $\dot{W}^n$, that is $u^n$ is the solution of
\begin{equation*}
\partial_{t} u^n_{t}(x) = \frac12 \Delta u^n_{t}(x) + u^n_{t}(x) \, \dot W^n_{t}(x), \quad t\in\R_{+}, x\in \R^{d},
\end{equation*}
understood in the classical Lebesgue sense. From here, we would like to \textit{define} the Stratonovich solution of \eqref{eq:she-intro} as the limit of $u^n$ as $n\to \infty$. The whole question behind this definition is of course to determine under which conditions such a convergence can indeed be guaranteed.

\smallskip

As long as $\dot{W}$ is not too irregular, this pathwise-type strategy can be successfully implemented through the so-called Young framework (see e.g. \cite[Section 5]{HHNT}). 
If one then wants to extend the above considerations to more irregular noises, some sophisticated procedures based on higher-order expansions and renormalization tricks must be involved. The so-called \textit{theory of regularity structures}, introduced by Hairer in \cite{hai-14}, provides us with both a convenient setting and powerful tools to address this extension issue. In the sequel, we will thus rely on Hairer's ideas to properly formulate and analyze the questions raised by equation \eqref{eq:she-intro} in a rough environment. 

\smallskip

This approach was already used in a similar fractional setting by one of the authors (see \cite{De16,De17}), so as to handle the one-dimensional non-linear heat model
\begin{equation}\label{eq:she-non-lin}
\partial_{t} u_{t}(x) = \frac12 \Delta u_{t}(x) + \si(x,u_{t}(x)) \, \dot W_{t}(x), \quad t\in [0,T], x\in \R,
\end{equation}
where $\si:\R \times \R \to \R$ is a smooth bounded function with compact support in its first variable, and $T$ is a small enough time. The latter assumptions clearly do not cover the model under consideration (i.e., equation \eqref{eq:she-intro}), and accordingly further work is required here. 

\smallskip

An important novelty to tackle in this situation is the \enquote{non-compactness} of the perturbation term $u \, \dot W$, as opposed to $\si(.,u) \, \dot W$ in \eqref{eq:she-non-lin} or to the torus framework that prevails in \cite{hai-14}. A natural idea to cope with this additional difficulty consists in the involvement of weighted topologies in the analysis. In the Young setting, such a weighted treatment of the model can be found in \cite[Section 5]{HHNT}. The basis of the corresponding analysis for the rough situation have been laid by Hairer and Labb{\'e} in \cite{HL}, with stochastic applications focusing on the white noise situation. 

\smallskip

Through the subsequent investigations, we propose to extend the application of the formalism of~\cite{HL} to the fractional situation, and thus provide a Stratonovich counterpart of the considerations of \cite{Ch18} regarding the Skorohod setting. In turn, the constructions below will be used as the starting point of the comparison procedure performed in \cite[Section 4]{CDOT}, and ultimately leading to new moment estimates for the solution of \eqref{eq:she-intro}.

\smallskip

Let us now specify the range of Hurst indexes $H_0,H_1,\ldots,H_d$, i.e. (morally) the range of regularities for $\dot{W}$, covered by the analysis in this paper. We recall first that the above-mentioned Young treatment of the model can be considered as long as $2H_0+H_1+\dots+H_d >d+1$ (see \cite[Section 5]{HHNT} or \cite[Section 5]{De16}). We here intend to focus on the next stage of the regularity-structure approach to the problem, which precisely corresponds to the condition
\begin{equation}\label{cond-h-strato-intro}
d+\frac23< 2H_0+H \leq d+1 , \quad \text{where} \ H:=\sum_{i=1}^d H_i\, . 
\end{equation}
The reason behind the restriction $2H_0+H>d+\frac23$ will become clear through the developments of Sections \ref{sec:setting} and \ref{sec:main-results} (see also Remark \ref{rk:extension-third-order} about possible extensions of the covering). Moreover, as we will observe it in the sequel, a drastic change of regime is to occur during the transition from the Young case to the \enquote{rough} case \eqref{cond-h-strato-intro}, with the involvement of a central second-order process above the fractional noise, the so-called $K$-L{\'e}vy area (see Definition \ref{defi:k-rp}). To some extent, and as suggested by our terminology, this change-of-regime phenomenon can be compared with the insight offered by the rough paths theory for the standard fractional differential equation
\begin{equation}\label{sde}
dY_t=\si(Y_t) dW_t\, ,
\end{equation}
where $W$ is a (standard) fractional Brownian motion of Hurst index $H\in (0,1)$. Indeed, it is a well-known fact that, when studying \eqref{sde}, the transition from the Young case $H>\frac12$ to the (first) rough case $\frac13<H\leq \frac12$ also involves the consideration of an additional (and crucial) L{\'e}vy-area term. 

\

Note that in order to avoid a long presentation of the numerous objects at the core of the original theory of regularity structures (model spaces, structure groups, regularity structures,...), we will rely in the sequel on the more direct \textit{$K$-rough paths} terminology introduced in \cite{De16}.

\

The rest of the paper is organized as follows. In Section \ref{sec:setting}, we introduce the framework of the analysis, and then rephrase the general well-posedness criterion of \cite{HL} using the $K$-rough paths terminology (Theorem \ref{theo:solution-map}). Our main result, namely the existence of such a $K$-rough path above the fractional noise, is presented in Section \ref{sec:main-results}, first in the space-time-noise situation (Section \ref{subsec:appli-st-fn}), then in the spatial-noise case (Section \ref{subsec:spatial-noise}). These statements will lead us to the desired Stratonovich solution of equation \eqref{eq:she-intro} (Definitions \ref{defi:strato-sol} and \ref{defi:strato-sol-spa}). The details of the construction of the fractional $K$-rough path in the space-time situation, resp. the spatial situation, will be provided in Section \ref{sec:stochastic-constructions}, resp. Section \ref{sec:stochastic-constructions-spa}. Finally, the appendix section contains the proofs of two useful technical results.

\section{Framework of the analysis}\label{sec:setting}

\subsection{General notation}\label{not:general}

For the sake of clarity, let us start by specifying a few pieces of notation that will be used throughout the study.

\smallskip

First, note that two different kinds of Fourier transforms on $\R^{d+1}$ will be involved in the sequel. Namely for a function $f(t,x)$ on $\R^{d+1}$, the Fourier transform on the full space-time domain $\R^{d+1}$ is
defined with the normalization
\begin{align}\label{def:Fourier transform} 
\mathcal{F} f (\eta, \xi) = \int_{\mathbb{R}^{d+1}} e^{- \imath\, 
   (t\eta+\xi\cdot x) } f (t,x) dtd x, 
   \end{align}
The analysis will also rely, at some point, on the spatial Fourier transform $\mathcal{F}^{\textsc s}$ given by
\begin{align}
\label{def:space Fourier transform} 
\mathcal{F}^{\textsc s} f (t, \xi) = \int_{\mathbb{R}^{d}} e^{- \imath\, 
   \xi\cdot x } f (t,x) d x. 
  \end{align}

Regarding the stochastic setting, we denote by $(\oom,\cf,\PP)$ the probability space related to $W$, with $\E$ for the related expected value. 
The heat kernel on $\R^{d}$ is denoted by $p_{t}(x)$, and recall that
\begin{equation}\label{eq:heat-kernel}
p_{t}(x) = \frac{1}{(2\pi t)^{d/2}} \, \exp\lp -\frac{|x|^{2}}{2t}  \rp.
\end{equation}
Also notice that the inner product of $a,b\in\R^d$ is written as $a\cdot b$ throughout the paper.

\smallskip

As mentioned in the introduction, we write $\bh$ for the vector of space Hurst parameters $(H_1,\dots,H_d)$, and denote the sum of these parameters as
\begin{equation}\label{defi:sum-h}
H = \sum_{j=1}^{d} H_{j}.
\end{equation}
Following the convention in \cite{hai-14}, the below considerations on the theory of regularity structures will occasionally appeal to the parabolic norm, defined for every $(s,x)\in \R^{d+1}$ as 
\begin{align}\label{parabolic norm}\|(s,x)\|_\scal:=\max\big( \sqrt{|s|}, |x_1|,\ldots,|x_d|\big) \, .\end{align}
Finally, we write $a\lesssim b$ to indicate that there exists an irrelevant constant $c$ such that $a\leq c b$.

\subsection{Weighted Besov topologies and $K$-rough paths}\label{subsec:topo-rs}

Our purpose in this section is to give an as-compact-as-possible presentation of the  regularity structures framework. As we mentioned above, the formalism is presented here in its weighted version (following \cite{HL}). Of course, we will only focus on its application to the dynamics under consideration,  that is to the model
\begin{equation}\label{eq:she-xi}
\left\{
\begin{array}{l}
\partial_{t} u  = \frac12 \Delta u + u \, \chi \, , \quad t\in [0,T] \, ,\,  x\in \R^{d} \, ,\\
u_0(x)=\psi(x) \ ,
\end{array}
\right.
\end{equation}
with $\chi$ a distribution of order $\al <0$ to be specified (at this point, the equation is only formal anyway). This customization of the theory will lead us to the introduction of a fundamental object at the core of the machinery: the $K$-rough path (see Definition~\ref{defi:k-rp} below).

The weights considered in the sequel have to satisfy a growth assumption which is summarized in the following definition.

\begin{definition}\label{def: weights}
A function $w:\R^d \to [1,\infty)$ is a \emph{weight on $\R^d$} if for every $M>0$, there exist $c_{1,M},c_{2,M}>0$ such that for every $x,y\in \R^d$ with $|x-y|\leq M$, one has
$$c_{1,M} \leq \frac{w(x)}{w(y)} \leq c_{2,M} \ .$$
\end{definition}
\noindent
Given a weight $w\in \R^d$, we will henceforth denote by $L^\infty_w(\R^{d+1})$ the space of functions defined by
\begin{align}\label{def: space-L^infty_w}
L^\infty_w(\R^{d+1})=\Big\{
f: \R^{d+1}\to\R;\  \mathrm{for\ all}\ T>0,
\  \sup_{(s,x)\in [-T,T]\times \R^{d}} \frac{|f_s(x)|}{w(x)} < \infty\Big\}.
\end{align}
We also write $\cac^0_w(\R^{d+1})$ for the set of continuous functions in $L^\infty_w(\R^{d+1})$.

Let us now turn to the definition of the (weighted) Besov-type spaces of distributions involved in Hairer's theory. Consider first the case of a positive order $\la\in (0,1)$:

\begin{definition}\label{def:besov-space-positive}
Let $w$ be a weight on $\R^d$. For every $\la\in (0,1)$, we will say that a function $\theta:\R^{d+1}\to \R$ belongs to $\cac^{\lambda}_w(\R^{d+1})$ if for every $T>0$,
$$\|\theta\|_{\la;T,w}:=\sup_{(s,x)\in [-T,T]\times \R^{d}} \frac{|\theta(s,x)|}{w(x)}+\sup_{((s,x), (t,y))\in D_{T,2}}\frac{|\theta(s,x)-\theta(t,y)|}{w(y) \|(s,x)-(t,y)\|_\scal^\la} \, < \, \infty \, ,$$
where we recall that the norm $\|\cdot\|_\scal$ is defined in \eqref{parabolic norm} and where the domain $D_{T,2}$ is defined by
\begin{multline}
D_{T,2} 
:=\Big\{((s,x),(t,y))\in\R^{d+1}\times\R^{d+1}; \, s,t\in[-T,T], {(s,x)\neq (t,y)\ \mathrm{and}  \ \|(s,x)-(t,y)\|_\scal \leq 2}\Big\}.
\end{multline}
\end{definition}

In order to define spaces of negative orders, we first need to recall the following notation for a scaling operator. Namely for all $\delta >0$, $(s,x),(t,y)\in \R^{d+1}$ and $\psi:\R^{d+1} \to \R$, denote
\begin{equation}\label{scaling-operator}
(\cs^\delta_{s,x} \psi)(t,y):=\delta^{-(d+2)} \vp\big(\delta^{-2}(t-s),\delta^{-1}(y-x)\big) \ .
\end{equation}
Also, for every $\ell \geq 0$, we will need to consider a specific set of compactly supported functions:
\begin{align}\label{cal B^l_s}
\cb_\scal^\ell=\{ \psi\in \cac^\ell(\R^{d+1});\  \mathrm{Supp}(\psi)\subset\cb_\scal(0,1)\ \mathrm{and}\ \|\psi\|_{\cac^\ell}\leq 1\},
\end{align}
where $\cac^\ell(\R^{d+1})$ refers to the space of $\ell$-times differentiable functions on $\R^{d+1}$, 
$$\|\psi\|_{\cac^\ell}:=\sup \big\{\|\partial_{x_{i_1}}\cdots \partial_{x_{i_k}}\psi \|_\infty, \ 0\leq k\leq \ell, \ i_1,\ldots,i_k\in \{1,\ldots,d+1\}\big\} \, ,$$
and $\cb_\scal(0,1)$ stands for the unit ball in $\R^{d+1}$ associated with the parabolic norm \eqref{parabolic norm}. Finally, we denote by $\cac^\ell_\infty(\R^{d+1})$ the space of $\ell$-times differentiable functions (on $\R^{d+1}$) with bounded derivatives, and define $\cd'_{\ell}(\R^{d+1})$ as the dual space of $\cac^\ell_\infty(\R^{d+1})$. With those additional notions in hand, we now give the definition of distributions with negative H\"{o}lder type continuity which is used in the sequel.
\begin{definition}\label{def:besov-space}
Let $w$ be a weight on $\R^d$ as given in Definition \ref{def: weights}. For every $\al <0$, we will say that a distribution $\chi\in \cd'(\R^{d+1})$ belongs to $\cac^{\al}_w(\R^{d+1})$ if it belongs to $\cd'_{2(d+1)}(\R^{d+1})$ and if for every $T>0$,
\begin{equation}\label{regu-psi-illu}
\|\chi\|_{\al;T,w}:=\sup_{(s,x)\in [-T,T]\times \R^{d}} \sup_{\vp\in \cb^{2(d+1)}_\scal} \sup_{\delta\in (0,1]} \frac{|\langle \chi,\cs^\delta_{s,x} \vp \rangle |}{\delta^\al w(x) } \, < \, \infty \, .
\end{equation}
\end{definition}

\begin{remark}
As can be seen from  \eqref{regu-psi-illu} we are considering topologies that are \enquote{localized} in time, and global, but \enquote{weighted}, in space. Besides, note that the choice of the regularity $2(d+1)$ in the condition $\chi\in\mathcal{D}'_{2(d+1)}$ is somewhat arbitrary. In fact, for the deterministic part of the analysis, we could replace this condition with $\chi\in \cd'_{r}(\R^{d+1})$ for any finite $r\geq 1$, as explained in \cite{hai-14}. The $2(d+1)$-regularity will only prove useful in the stochastic constructions of Section \ref{sec:stochastic-constructions} (see for instance Lemma \ref{lem:fouri-psi-prop}).
\end{remark}

The following topological spaces, which somehow correspond to \enquote{lifted versions} of $\cac^{\al}_w(\R^{d+1})$, will later accommodate the central $K$-rough paths:

\begin{definition}\label{def:besov-space-2}
Let $w$ be a weight on $\R^d$. For every $\al <0$, we say that a map $\zeta:\R^{d+1} \to  \cd'(\R^{d+1})$ belongs to $\pmb{\cac}_w^\al(\R^{d+1})$ if for every $(s,x)\in \R^{d+1}$, $\zeta_{s,x}$ belongs to $\cd'_{2(d+1)}(\R^{d+1})$ and if, for every $T>0$,
$$\|\zeta\|_{\al;T,w}:=\sup_{(s,x)\in [-T,T]\times \R^{d}} \sup_{\vp\in \cb^{2(d+1)}_\scal} \sup_{\delta\in (0,1]} \frac{|\langle \zeta_{s,x},\cs^\delta_{s,x} \vp \rangle |}{\delta^\al w(x)}\, < \, \infty \, ,$$
where the sets $\mathcal{B}^l_\scal$ are given by \eqref{cal B^l_s}.
\end{definition}

We still need one last technical ingredient in the procedure: the definition of a \emph{localized heat kernel}, which essentially transcribes the singular behavior of the (global) heat kernel around $(0,0)$. 

\begin{definition}\label{defi:loc-heat-ker}
We call a \emph{localized heat kernel} any function $K:\R^{d+1} \backslash \{0\}\to \R$ satisfying the following conditions:

\noindent
\emph{(i)} It holds that $p_s(x)=K(s,x)+R(s,x)$, for some \enquote{remainder} $R\in \cac^\infty(\R^{d+1})$, where we recall that the heat kernel $p$ is defined by \eqref{eq:heat-kernel}.

\noindent
\emph{(ii)} $K(s,x)=0$ as soon as $s\leq 0$.

\noindent
\emph{(iii)} There exists a smooth function $K_0:\R^{d+1} \to \R$ with support in $[-1,1]^{d+1}$ such that for every non-zero $(s,x) \in \R^{d+1}$, one has 
\begin{equation}\label{decompo-k}
K(s,x)=\sum_{\ell\geq 0} 2^{-2\ell} (\cs_{0,0}^{2^{-\ell}}K_0)(s,x) \quad \text{and} \quad R(s,x)=\sum_{\ell < 0} 2^{-2\ell} (\cs_{0,0}^{2^{-\ell}}K_0)(s,x) \ .
\end{equation}
\end{definition}

We are finally in a position to introduce the key object of the machinery, namely a distribution in the second chaos of the noise $\chi$ which plays the role of the L\'{e}vy area in our context.

\begin{definition}\label{defi:k-rp}
Let $w$ be a weight on $\R^d$ (see Definition \ref{def: weights}), let $K$ be a localized heat kernel (see Definition \ref{defi:loc-heat-ker}) and consider $\al <0$. Also, fix $\chi\in \cac^{\al}_w(\R^{d+1})$. We call an \emph{$(\al,K)$-L{\'e}vy area above $\chi$} (\emph{for the weight $w$}) any map $\mathcal{A} : \R^{d+1} \to \cd'(\R^{d+1})$ satisfying the two following conditions.

\noindent
(i) \emph{$K$-Chen relation:} For all $(s,x),(t,y)\in \R^{d+1}$, 
$$\mathcal{A}_{s,x}-\mathcal{A}_{t,y}=[(K\ast \chi)(t,y)-(K\ast \chi)(s,x)] \cdot \chi \ ,$$
where the notation $\ast$ refers to the space-time convolution.

\noindent
(ii) \emph{Besov regularity:} $\mathcal{A}$ belongs to $\pmb{\cac}^{2\al+2}_w(\R^{d+1})$, where the space $\pmb{\cac}^{2\al+2}_w$ is introduced in Definition~\ref{def:besov-space-2}.

\noindent
We call \emph{$(\alpha,K)$-rough path} above $\chi$ (for the weight $w$) any pair $\uchi=(\chi,\bchi^{\mathbf{2}})$ where $\chi \in \cac^{\al}_w(\R^{d+1})$ and $\bchi^{\mathbf{2}}$ is an $(\al,K)$-L{\'e}vy area above $\chi$ (for the weight $w$). 
We denote by $\mathcal{E}^{K}_{\al;w}$ the set of such $(\alpha,K)$-rough paths (for the weight $w$). If $\uchi=(\chi,\bchi^{\mathbf{2}}),\pmb{\zeta}=(\zeta,\zeta^{\mathbf{2}}) \in \mathcal{E}^{K}_{\al;w}$, we set
$$\|\uchi,\pmb{\zeta}\|_{\al;T,w}:=\|\chi-\zeta\|_{\al;T,w}+\|\bchi^{ \mathbf{2}}-\zeta^{ \mathbf{2}}\|_{2\al+2;T,w} \, .
$$
A global distance on $\mathcal{E}^{K}_{\al;w}$ is then given by
\begin{equation}\label{def: global norm-rp}
d_{\al;w}(\uchi,\pmb{\zeta})=\sum_{k\geq 1}2^{-k} \frac{\|\uchi,\pmb{\zeta}\|_{\al;k,w}}{1+\|\uchi,\pmb{\zeta}\|_{\al;k,w}} \ .
\end{equation}
\end{definition}

By mimicking the arguments of the proof of \cite[Proposition 3.1]{De16}, we immediately deduce the following completeness property:
\begin{lemma}\label{complet}
For every weight $w$ on $\R^d$, every localized heat kernel $K$ and every $\al <0$, $(\mathcal{E}^{K}_{\al;w},d_{\al;w})$ is a complete metric space.
\end{lemma}

Let us complete Definition \ref{defi:k-rp} with two fundamental remarks, that often turn out to be essential in the application of the theory.

\begin{remark}\label{rk:levy-area-smooth}
Recall that the space $L^\infty_w(\R^{d+1})$ is defined by \eqref{def: space-L^infty_w}.  In the \enquote{regular} situation where $\chi\in L^\infty_w(\R^{d+1})$, there exists a straightforward \textit{canonical $K$-L{\'e}vy area above $\chi$} (for the weight $w^2$)  given by the formula 
\begin{equation}\label{levy-area-smoo}
\bchi^{\mathbf{2}}_{s,x}(t,y):=[(K\ast \chi)(t,y)-(K\ast \chi)(s,x)]\cdot \chi(t,y) \ ,
\end{equation}
where we recall that $\ast$ refers to space-time convolution in this setting. The resulting \textit{canonical $K$-rough path} will be our standard reference in approximation (or continuity) results. The situation can here be compared with Lyons' rough paths theory, where (classical) rough paths are often obtained as the limit of the canonical rough path given by the set of iterated integrals. 
\end{remark}

\begin{remark}
Starting from a $K$-L{\'e}vy area $\bchi^{\mathbf{2}}$, any constant $c$ gives rise to another $K$-L{\'e}vy area by setting $\widehat{\bchi}^{\mathbf{2}}_{s,x}(t,y):=\bchi^{\mathbf{2}}_{s,x}(t,y)-c$, which  paves the way toward renormalization tricks. In the sequel, we will use the notation 
\begin{align}\label{def:renormalized noise}\text{Renorm}((\chi,\bchi^{\mathbf{2}}),c):=(\chi,\bchi^{\mathbf{2}}-c)\end{align}
for such elementary renormalization.
\end{remark}

\subsection{A general solution map}
With the above setting and notation in hand, the following \enquote{black box} statement about equation \eqref{eq:she-xi} can now be derived from a slight adaptation of the considerations and results of \cite{HL}:

\begin{theorem}\label{th:solution map}[Solution map]\label{theo:solution-map}
Fix an arbitrary time horizon $T>0$ and a parameter $\al \in (-\frac43,-1)$. Then there exist a localized heat kernel $K$, two weights $w_1,w_2$ on $\R^d$ (that depend on $T$),  and a \enquote{solution} map
\begin{equation}\label{solut-map}
\Phi=\Phi^{K,T}_{\al,w_1,w_2} : \mathcal{E}^{K}_{\al;w_1} \times L^\infty(\R^d) \longrightarrow\  L^\infty([0,T]; L^\infty_{w_2}(\R^d)),
\end{equation}
where $\mathcal{E}^K_{\al;w_1}$ is introduced in Definition \ref{defi:k-rp} and $L^\infty_w$ is given by \eqref{def: space-L^infty_w}. The map $\Phi$ is such that the following properties are satisfied:

\smallskip

\noindent
\emph{(i) Weights.} One has $w_1(x)=(1+|x|)^{\ka_1}$ and $w_2(x)=e^{\ka_2 (1+|x|)}$, for some $\ka_1,\ka_2 >0$. 

\smallskip 

\noindent
\emph{(ii) Consistency.} Assume $\chi\in L^\infty_{w_1^{1/2}}(\R^{d+1})$ and $\uchi\in \mathcal{E}^{K}_{\al;w_1}$ is the canonical $K$-rough path above $\chi$ with L{\'e}vy-area term defined along \eqref{levy-area-smoo}. Then for any $\psi\in L^\infty(\R^d)$ one has $\Phi(\uchi,\psi)=u$, where $u$ is the classical solution on $[0,T]$ of equation \eqref{eq:she-xi}.

\smallskip

\noindent
\emph{(iii) Renormalization.} As in item (ii), consider $\chi\in L^\infty_{w_1^{1/2}}(\R^{d+1})$ and its canonical $K$-rough path $\uchi$. For an initial condition $\psi\in L^\infty(\R^d)$ and $c\in \R$, set $\widehat{u}=\Phi(\text{Renorm}(\uchi,c),\psi)$, where $\text{Renorm}(\uchi,c)$ is defined by \eqref{def:renormalized noise}. Then $\widehat{u}$ is the classical solution on $[0,T]$ of the equation 
\begin{equation*}
\left\{
\begin{array}{l}
\partial_{t} \widehat{u}  = \frac12 \Delta \widehat{u} + \widehat{u}\, \chi-c\, \widehat{u} \, , \quad t\in [0,T],\,  x\in \R^{d} \, ,\\
u_0(x)=\psi(x) \ .
\end{array}
\right.
\end{equation*}

\smallskip

\noindent
\emph{(iv) Continuity.} Let $(\uchi,\psi)\in \mathcal{E}^{K}_{\al,w_1} \times L^\infty(\R^d)$ and let $(\uchi^n,\psi^n)\in \mathcal{E}^{K}_{\al;w_1} \times L^\infty(\R^d)$ be a sequence such that
$$d_{\al,w_1}(\uchi^n,\uchi) \to 0 \quad \text{and} \quad \lVert \psi^n-\psi\rVert_{L^\infty(\R)} \to 0  \ ,$$
where $d_{\al,w_1}$ is the distance introduced in \eqref{def: global norm-rp}.
Then $\Phi(\uchi^n,\psi^n)$ converges to $\Phi(\uchi,\psi)$ in the space $L^\infty([0,T];L^\infty_{w_2}(\R^d))$.
\end{theorem}

\begin{remark}
We are aware that the corresponding results in \cite{HL} are actually expressed in terms of (weighted) \textit{models} and \textit{structure group}, following the general terminology of \cite{hai-14}. However, the transition from our (lighter) notion of an $(\al,K)$-rough path to a \textit{regularity structure} (that is, a model together with a structure group) is a matter of elementary considerations, as detailed in \cite[Proposition 2.5]{De16}. The only technical point requiring some attention is the control of $K\ast \chi$, as an element of $\cac^{\al+2}_{w_1}(\R^{d+1})$, in terms of $\chi\in \cac^{\al}_{w_1}(\R^{d+1})$, for $\al\in (-\frac43,-1)$. In fact, following the lines of the proof of \cite[Lemma 2.2]{De16}, one can easily check that for every weight $w$ on $\R^d$, every $\al\in (-2,-1)$, every $\chi\in \cac_w^\al(\R^{d+1})$ and every time $T>0$, one has
\begin{equation}\label{bou-k-ast}
\|K\ast \chi\|_{\al+2;T,w} \lesssim \|\chi\|_{\al;T,w} \, ,
\end{equation}
which precisely corresponds to the control we need in order to justify this transition.
\end{remark}

\section{Main results}\label{sec:main-results}

We now go back to the stochastic setting and to the consideration of a fractional noise $\chi:=\dot{W}$ in equation \eqref{eq:she-xi}. In other words, we go back here to the analysis of \eqref{eq:she-intro}. With the result of Theorem~\ref{theo:solution-map} in mind, the strategy toward the desired Stratonovich solution is clear: we need to construct a $K$-rough path above $\dw$ in the almost sure sense, preferably as the limit of some (renormalized) canonical $K$-rough path (for the continuity property $(iv)$ in Theorem \ref{th:solution map} to hold).

\smallskip

First, we will proceed to the detailed presentation of our existence result in the situation where $\dot{W}$ is the space-time fractional noise defined by \eqref{eq:def-gamma} (for $(H_0,\bh)$ satisfying \eqref{cond-h-strato-intro}). Then we will review the main steps of the construction in the (easier) situation where $\dot{W}$ is only a \textit{spatial} fractional noise.

\subsection{Application to a space-time fractional noise}\label{subsec:appli-st-fn}

Let $\dot{W}$ be the noise defined by \eqref{eq:def-gamma}, for some Hurst index $H_0\in (0,1)$ in time and $\bh=(H_1,\ldots,H_d)\in (0,1)^{d}$ in space. Let us recall that $\dot{W}$ can also be seen as the derivative of a space-time fractional Brownian motion $W$, that is $\dot{W}=\partial_t \partial_{x_1} \cdots \partial_{x_d} W$. As a consequence, one can easily define a smooth approximation $\dw^n$ of $\dw$ by using a standard mollifying procedure. 

\smallskip

To be more specific, we define the approximated noise $\dw^n$ by $\dw^0:=0$ and for $n\geq 1$, 
\begin{align}\label{def:approximate noise}
\dw^n:=\partial_t  \partial_{x_1} \cdots \partial_{x_d} W^n, \quad \text{where} \ W^n:=\rho_n \ast W \ \text{and} \ \rho_n(s,x):=2^{n(d+2)}\rho(2^{2n} s,2^n x),
\end{align}
for some mollifier $\rho: \R^{d+1} \to \R_+$ satisfying the following (natural) assumptions:

\noindent
\textbf{Assumption ($\rho$).} We consider a smooth, even, and $L^1(\R^{d+1})$ function $\rho :\R^{d+1} \to \R_+$. In addition we suppose that $\rho$ satisfies 

\smallskip

\noindent
$(i)$  $\int_{\R^{d+1}} \rho(s,x) \, ds dx=1$. 

\smallskip

\noindent
$(ii)$ The Fourier transform $\mathcal F{\rho}$ is Lipschitz.

\smallskip

\noindent
$(iii)$ For every $(\tau_0,\tau_1,\ldots,\tau_d) \in [0,1]^{d+1}$, the following upper bound holds true for every $(\la, \xi)\in\R^{d+1}$,
\begin{equation}\label{asympt-rho}
|\mathcal F{\rho}(\la,\xi)| \leq c_\tau |\la|^{-\tau_0} \, \prod_{i=1}^{d}|\xi_{i}|^{-\tau_{i}}  .
\end{equation}

\smallskip

\begin{remark}\label{remark on weights}
Assumption $(\rho)$ is  trivially satisfied by any smooth, even and compactly-supported function $\rho :\R^{d+1}\to\R_+$ such that $\int_{\R^{d+1}} \rho(s,x) \, ds dx=1$. These conditions also cover the mollifying function considered in \cite[Section 3.2]{HHNT} or in \cite[Section 5]{HN}, that is $\rho(s,x):=\vp(s) p_1(x)$, where $\vp:=\1_{[0,1]}$ and $p_1$ refers to the Gaussian density \eqref{eq:heat-kernel} at time $1$. Last but not least, Assumption $(\rho)$ is satisfied by the mollifier considered in the Skorohod analysis of \cite[Section 3]{CDOT}, that is $\rho(s,x):=p_1(s) p_1(x)$. The latter choice will become our standard reference in the subsequent Definition \ref{defi:strato-sol}.
\end{remark}

Once endowed with the approximation $\dw^n$, let us consider the canonical $K$-rough path $\pmb{W}^{n}:=(\dw^{n},\mathbf{W}^{\mathbf{2},n})$, defined along Remark \ref{rk:levy-area-smooth}. Namely we set
\begin{equation}\label{notation-xi-2}
\bw^{\mathbf{2},n}_{s,x}(t,y):=\ci^n_{s,x}(t,y)\cdot \dw^{n}(t,y) \, , 
\end{equation}
where
\begin{equation}\label{notationci-n}
\ci^n_{s,x}(t,y):=(K\ast \dw^{n})(t,y)-(K\ast \dw^{n})(s,x) \, .
\end{equation}

With this setting in hand, our main statement will consist in a convergence property for the (suitably renormalized) sequence $\pmb{W}^{n}:=(\dw^{n},\mathbf{W}^{\mathbf{2},n})$.  The statement will appeal, among other things, to the following technical result (the proof of which is postponed to Section \ref{subsec:proof-renorm-cstt-tech}).
\begin{lemma}\label{lem:cstt-renorm}
Let $\rho$ be a mollifier satisfying Assumption $(\rho)$, and let $H_0\in (0,1), \bh=(H_1,\ldots,H_d)\in (0,1)^{d}$ be such that 
\begin{equation}\label{defi:h-scal}
2H_0+H \leq d+1 \, ,
\end{equation}
where the notation $H$ has been introduced in \eqref{defi:sum-h}. Recall that the heat kernel $p$ is defined by \eqref{eq:heat-kernel}. Let us set from now on
\begin{equation}\label{notation:cn-h}
\cn_{H_0,\bh}(\la,\xi):=\frac{1}{|\la|^{2H_0-1}} \prod_{i=1}^d \frac{1}{|\xi_i|^{2H_i-1}} \, ,
\end{equation}
namely $c_0c_\bh\,\cn_{H_0,\bh}$ is the Fourier transform of the mesure $\ga_0\otimes\ga$ introduced in \eqref{eq:rep-gamma-fourier}.
Then, for every fixed $c>0$, the integral
\begin{equation}\label{eq:cj-c-n-border}
\int_{|\la|+|\xi|^2 \geq c}  |\mathcal F{\rho}(\la,\xi)|^2  \mathcal F{p}(\la,\xi) \cn_{H_0,\bh}(\la,\xi)\, d\la d\xi 
\end{equation}
is finite, and when $2H_0+H<d+1$, it even holds that
\begin{equation}\label{eq:cj-c-n}
\cj_{\rho,H_0,\bh}:=\int_{\R^{d+1}}  |\mathcal F{\rho}(\la,\xi)|^2  \mathcal F{p}(\la,\xi) \cn_{H_0,\bh}(\la,\xi)\, d\la d\xi \, < \, \infty \, .
\end{equation}
\end{lemma}

For simplicity, let us set from now on $c_{H_0,\bh}:=c_{H_0}c_{\bh}$, where $c_{H_0},c_{\bh}$ are the constants defined in~\eqref{cstt-c-h}. We are now ready to state the result about the existence of a $K$-rough path above our noise.
\begin{theorem}\label{theo:conv-k-rp}
Let $\rho$ be a mollifier satisfying Assumption $(\rho)$.  Consider Hurst parameters  $H_0\in (0,1)$ and $\bh=(H_1,\ldots,H_d)\in (0,1)^{d}$. 
We strengthen condition \eqref{defi:h-scal} in the following way: 
\begin{align}\label{strengthened 4.17}
d+\frac12< 2H_0+H \leq d+1,
\end{align}
where we recall that $H$ is given by \eqref{defi:sum-h}. In this setting, fix $\al \in \R$ such that 
\begin{align}\label{in th: condition alpha}
 \al < -(d+2)+2H_0+H.\end{align}
For $n\geq1$, define $\dw^n$ as in \eqref{def:approximate noise} and set 
\begin{align}\label{hat W^n}
\widehat{\pmb{W}}^{n}:=\mathrm{Renorm}(\uw^{n},\mathfrak{c}_{\rho,H_0,\bh}^{(n)}),  
\end{align}
with
\begin{equation}\label{renormal}
\mathfrak{c}_{\rho,H_0,\bh}^{(n)}:=
\begin{cases}
c_{H_0,\bh}^2 \, 2^{2n(d+1-(2H_0+H))} \cj_{\rho,H_0,\bh}   & \text{if} \ 2H_0+H<d+1\\
&\\
c_{H_0,\bh}^2 \int_{|\la|+|\xi|^2\geq 2^{-2n}}  |\mathcal F{\rho}(\la,\xi)|^2  \mathcal F{p}(\la,\xi) \cn_{H_0,\bh}(\la,\xi)\, d\la d\xi& \text{if} \ 2H_0+H=d+1
\end{cases}
\end{equation}
where the operator Renorm is introduced in \eqref{def:renormalized noise} and the quantity $\cj_{\rho,H_0,\bh}$ is defined in \eqref{eq:cj-c-n}. 

\smallskip

Then for any weight $w(x):=(1+|x|)^\ka$ with $\ka>0$ and for the distance $d_{\alpha,w}$ given by \eqref{def: global norm-rp}, there exists an $(\al,K)$-rough path $\widehat{\uw}$ such that almost surely
\begin{equation}\label{desired-convergence}
\lim_{n\to\infty}d_{\al,w}(\widehat{\uw}^{n},\widehat{\uw}) = 0.
\end{equation}
\end{theorem}

For the sake of clarity, we have postponed the (long technical) proof of Theorem \ref{theo:conv-k-rp} to Section \ref{sec:stochastic-constructions}.

\smallskip

Now, by combining the deterministic result of Theorem \ref{theo:solution-map} with the stochastic construction of Theorem \ref{theo:conv-k-rp}, we derive the desired Stratonovich interpretation of equation \eqref{eq:she-intro}:

\begin{definition}\label{defi:strato-sol}
Let $\rho$ be the weight given by $\rho(s,x):=p_1(s)p_1(x)$ as considered in Remark~\ref{remark on weights}.  Let $(H_0, \bh)\in (0,1)^{d+1}$ be a vector of Hurst parameters such that 
\begin{equation}\label{cond-h-strato}
d+\frac23< 2H_0+H \leq d+1 .
\end{equation}
Besides, fix $\al \in \R$ such that 
$$-\frac43 < \al < -(d+2)+2H_0+H \, ,$$
as well as an arbitrary time horizon $T>0$ and an initial condition $\psi\in L^\infty(\R^d)$. Then, using the notations of Theorem \ref{theo:solution-map} and Theorem \ref{theo:conv-k-rp}, we call $u:=\Phi^{K,T}_{\al,w_1,w_2}(\widehat{\uw},\psi)$ the \emph{renormalized Stratonovich solution} of equation \eqref{eq:she-intro}, with initial condition $\psi$. In particular, $u$ is the (almost sure) limit, in $L^\infty([0,T]\times \R^d)$, of the sequence $(u^{n})_{n\geq 1}$ of classical solutions of the equation 
\begin{equation}\label{eq:u-square-n}
\left\{
\begin{array}{l}
\partial_{t} u^{n}  = \frac12 \Delta u^{n} + u^{n}\, \dw^n-\mathfrak{c}_{\rho,H_0,\bh}^{(n)}\, u^{n} \, , \quad t\in [0,T],\,  x\in \R^{d} \, ,\\
u^{n}_0(x)=\psi(x) \ .
\end{array}
\right.
\end{equation}
\end{definition}

\

Let us complete the above Definition \ref{defi:strato-sol} with two comments.

\begin{remark}
Observe that the assumptions on $H_0,\bh$ in \eqref{cond-h-strato} are more restrictive than those in Theorem \ref{theo:conv-k-rp}. This stronger restriction actually stems from Theorem \ref{theo:solution-map}, which requires $\al$ to be strictly larger than $-\frac43$.
\end{remark}

\begin{remark}\label{rk:extension-third-order}
As the reader might expect it, the extension of the result of Theorem \ref{theo:solution-map} to any $\al >-\frac32$ (and not only $\al>-\frac43$) is in fact possible, at the price of an additional \enquote{third-order} elements (on top of $\chi$ and $\bchi^{\mathbf{2}}$) in the definition of a $K$-rough path (see \cite[Definition 2.7]{De17} for details when $d=1$). Therefore, applying this extension to our stochastic model would require us to construct additional \enquote{third-order} processes above the fractional noise.
This strategy has been implemented in~\cite{De17} for $d=1$, and when working with the \enquote{compact-in-space} topologies derived from the analysis of \eqref{eq:she-non-lin}. We firmly believe that the constructions of \cite{De17} could be extended to the current setting, that is to any dimension $d\geq1$ and to the whole space $\R^d$, at the price of highly sophisticated computations.
\end{remark}

\smallskip

Let us finally conclude the section with the exhibition of an asymptotic equivalence for the constant $\mathfrak{c}_{\rho,H_0,\bh}^{(n)}$ in \eqref{hat W^n}, \textit{in the limit case $2H_0+H=d+1$} (the proof of this statement can be found in Section \ref{subsec:proof-asymp}).

\begin{proposition}\label{prop:asymp-renorm-cstt}
In the setting of Theorem \ref{theo:conv-k-rp}, assume that $2H_0+H=d+1$. Then, as $n$ tends to infinity, it holds that
\begin{equation}\label{decompo-cstt-border}
\mathfrak{c}_{\rho,H_0,\bh}^{(n)}=n\cdot C_{H_0,\bh}+O(1),
\end{equation}
for some constant $C_{H_0,\bh}$ independent of $\rho$.
\end{proposition}

Thus, when compared to the behavior of $\mathfrak{c}_{\rho,H_0,\bh}^{(n)}$ as $2H_0+H<d+1$ (see \eqref{renormal}), the expansion~\eqref{decompo-cstt-border} clearly emphasizes the specificity of the border case $2H_0+H=d+1$ in the analysis of the problem.

\subsection{Application to a spatial fractional noise}\label{subsec:spatial-noise}

\

\smallskip

We now would like to specialize the previous results to a {\it spatial} fractional noise. In other words, we consider here $\{W^{\bh}(x), \, x\in \R^d\}$ a spatial fractional Brownian motion of Hurst index $\bh\in (0,1)^d$ and set 
\begin{equation}\label{defi-spatial-noise}
\dw:=\partial_{x_1}\cdots \partial_{x_d} W^{\bh}.
\end{equation}
In many situations, it is known that, at least at a formal level, the transition from a space-time fractional noise to a spatial fractional noise essentially reduces to \enquote{taking $H_0=1$}. Our aim in the sequel to fully justify this phenomenon in the situation we are interested in, that is the study of equation \eqref{eq:she-intro}. To this end, we propose to review the successive steps of the analysis provided in Section \ref{subsec:appli-st-fn} and examine the corresponding results in the spatial situation.

\smallskip

Thus, as a first step, we introduce a smooth approximation $\dw^n$ of $\dw$ obtained through a general mollifying procedure. That is, we define the approximated noise $\dw^n$ by $\dw^0:=0$ and for $n\geq 1$, 
\begin{align}\label{def:approximate noise-spa}
\dw^n(s,x)=\dw^n(x):=\big(\partial_{x_1} \cdots \partial_{x_d} W^n\big)(x), \quad  W^n:=\rho_n \ast W^{\bh},\quad \rho_n(x):=2^{dn}\rho(2^n x),
\end{align}
for some mollifier $\rho: \R^{d} \to \R_+$ satisfying the following assumptions (remember that the notation $\mathcal{F}^{\textsc s}$ refers to the spatial Fourier transform, along \eqref{def:space Fourier transform}):

\smallskip

\noindent
\textbf{Assumption ($\rho$).} We consider a smooth, even, and $L^1(\R^{d})$ function $\rho :\R^{d} \to \R_+$. In addition we suppose that $\rho$ satisfies

\smallskip

\noindent
$(i)$  $\int_{\R^{d}} \rho(x) \, dx=1$.

\smallskip

\noindent
$(ii)$ The Fourier transform $\mathcal{F}^{\textsc s}{\rho}$ is Lipschitz.

\smallskip

\noindent
$(iii)$ For every $(\tau_1,\ldots,\tau_d) \in [0,1]^{d}$, the following upper bound holds true for every $\xi\in\R^{d}$,
\begin{equation}\label{asympt-rho-spa}
|\mathcal{F}^{\textsc s}{\rho}(\xi)| \leq c_\tau  \prod_{i=1}^{d}|\xi_{i}|^{-\tau_{i}}  .
\end{equation}

\

The canonical $K$-rough path $(\pmb{W}^{n})_{n\geq 1}:=(\dw^{n},\mathbf{W}^{\mathbf{2},n})_{n\geq 1}$ above $\dw^n$ can here be written as
\begin{equation}\label{notation-xi-2-spa}
\bw^{\mathbf{2},n}_{s,x}(t,y)=\bw^{\mathbf{2},n}_{x}(y):=\ci^n_{x}(y)\cdot \dw^{n}(y) \, , 
\end{equation}
where
\begin{equation}\label{notationci-n-spa}
\ci^n_{x}(y):=(\tilde{K}\ast \dw^n)(y)-(\tilde{K}\ast \dw^n)(x) \, ,
\end{equation}
with
\begin{equation}\label{kernel-k-tilde}
\tilde{K}(x):=\int_0^\infty ds \, K(s,x) \, .
\end{equation}
It is worth noting that, owing to the very definition of $K$ (see Definition \ref{defi:loc-heat-ker}), the latter integral is indeed finite (for every fixed $x\in \R^d$), and also that $\kti\in L^1(\R^d)$.

\smallskip

The spatial counterpart of the preliminary Lemma \ref{lem:cstt-renorm} now reads as follows (the proof of this property can be shown with similar estimates to the ones in Section \ref{subsec:proof-renorm-cstt-tech}).
\begin{lemma}\label{lem:cstt-renorm-spa}
Let $\rho:\R^d \to \R$ be a mollifier satisfying Assumption $(\rho)$, and let $\bh=(H_1,\ldots,H_d)\in (0,1)^{d}$ be such that 
\begin{equation}\label{defi:h-scal-spa}
H < d-1 \, ,
\end{equation}
where the notation $H$ has been introduced in \eqref{defi:sum-h}. Let us set from now on
\begin{equation}\label{notation:cn-h-spa}
\cn_{\bh}(\xi):= \prod_{i=1}^d \frac{1}{|\xi_i|^{2H_i-1}} \, ,
\end{equation}
namely $c_\bh\,\cn_{\bh}$ is the Fourier transform of the measure $\ga_{\bh}$ introduced in \eqref{eq:rep-gamma-fourier}. Besides, recall that the heat kernel $p$ is defined by \eqref{eq:heat-kernel}.
Then the following integral is finite:
\begin{equation}\label{eq:cj-c-n-spa}
\cj_{\rho,\bh}:=\int_{\R^{d}}  |\mathcal{F}^{\textsc s}{\rho}(\xi)|^2   \cn_{\bh}(\xi)\bigg( \int_0^\infty ds\, \mathcal{F}^{\textsc s} p_s(\xi)\bigg)\, d\xi .
\end{equation}
\end{lemma}

\smallskip

We are now in a position to present the (expected) counterpart of Theorem \ref{theo:conv-k-rp} for the spatial situation. 

\begin{theorem}\label{theo:conv-k-rp-spa}
Let $\rho:\R^d \to \R$ be a mollifier satisfying Assumption $(\rho)$, \textcolor{blue}{and fix $d\geq 2$}. 
Let $\bh=(H_1,\ldots,H_d)\in (0,1)^{d}$ be a vector of Hurst parameters such that
\begin{align}\label{strengthened 4.17-spa}
d-\frac32< H \leq  d-1,\
\end{align}
where we recall that $H$ is given by \eqref{defi:sum-h}. In this setting, fix $\al < H-d$.

For $n\geq1$, define $\dw^n$ as in \eqref{def:approximate noise} and set $\widehat{\pmb{W}}^{n}:=\mathrm{Renorm}(\uw^{n},\mathfrak{c}_{\rho,\bh}^{(n)})$, 
with
\begin{equation}\label{renormal-spa}
\mathfrak{c}_{\rho,\bh}^{(n)}:=
\begin{cases}
2^{2n(d-H-1)}c_{\bh}^2 \cj_{\rho,\bh}  & \text{if} \ H<d-1\\
&\\
c_{\bh}^2 \int_{|\xi|\geq 2^{-n}}  |\mathcal{F}^{\textsc s}{\rho}(\xi)|^2   \cn_{\bh}(\xi)\bigg( \int_0^\infty ds\, \mathcal{F}^{\textsc s} p_s(\xi)\bigg)\, d\xi& \text{if} \ H=d-1
\end{cases}
\end{equation}
where the constant $c_{\bh}$ is defined in \eqref{cstt-c-h} and the quantity $\cj_{\rho,\bh}$ in \eqref{eq:cj-c-n-spa}.

Then for any weight $w(x):=(1+|x|)^\ka$ with $\ka>0$ and for the distance $d_{\alpha,w}$ given by \eqref{def: global norm-rp}, there exists an $(\al,K)$-rough path $\widehat{\uw}$ such that almost surely
\begin{equation}\label{desired-convergence-spa}
\lim_{n\to\infty}d_{\al,w}(\widehat{\uw}^{n},\widehat{\uw}) = 0.
\end{equation}
\end{theorem}

\begin{proof}
See Section \ref{sec:stochastic-constructions-spa} for a survey of the adaptations to be made with respect to the arguments used in the proof of Theorem \ref{theo:conv-k-rp}.
\end{proof}

By injecting the $K$-rough path constructed in Theorem \ref{theo:conv-k-rp-spa} into the general wellposedness statement of Theorem \ref{theo:solution-map}, we immediately derive the following spatial equivalent of Definition \ref{defi:strato-sol}.

\begin{definition}\label{defi:strato-sol-spa}
Let $\rho$ be the weight given by $\rho(x):=p_1(x)$.  Let $\bh\in (0,1)^{d}$ be a vector of Hurst parameters such that 
\begin{equation}\label{cond-h-strato-spa}
d-\frac43< H \leq d-1 \, .
\end{equation}
Besides, fix $\al <0$ such that $-\frac43 < \al < H-d$, as well as an arbitrary time horizon $T>0$ and an initial condition $\psi\in L^\infty(\R^d)$. Then, using the notations of Theorem \ref{theo:solution-map} and Theorem \ref{theo:conv-k-rp-spa}, we call $u:=\Phi^{K,T}_{\al,w_1,w_2}(\widehat{\uw},\psi)$ the \emph{renormalized Stratonovich solution} of equation \eqref{eq:she-intro}, with initial condition $\psi$. In particular, $u$ is the (almost sure) limit, in $L^\infty([0,T]\times \R^d)$, of the sequence $(u^{n})_{n\geq 1}$ of classical solutions of the equation 
\begin{equation}\label{eq:u-square-n-spa}
\left\{
\begin{array}{l}
\partial_{t} u^{n}  = \frac12 \Delta u^{n} + u^{n}\, \dw^n-\mathfrak{c}_{\rho,\bh}^{(n)}\, u^{n} \, , \quad t\in [0,T],\,  x\in \R^{d} \, ,\\
u^{n}_0(x)=\psi(x) \ .
\end{array}
\right.
\end{equation}
\end{definition}

\smallskip

In a similar way to Proposition \ref{prop:asymp-renorm-cstt} (and using similar proof arguments), we can finally show that the constant $\mathfrak{c}_{\rho,\bh}^{(n)}$ in \eqref{renormal-spa} adopts a specific behaviour when $H=d-1$.

\begin{proposition}\label{prop:asymp-renorm-cstt-spa}
In the setting of Theorem \ref{theo:conv-k-rp-spa}, assume that $H=d-1$. Then, as $n$ tends to infinity, it holds that
\begin{equation}\label{decompo-cstt-border-spa}
\mathfrak{c}_{\rho,\bh}^{(n)}=n\cdot C_{\bh}+O(1),
\end{equation}
for some constant $C_{\bh}$ independent of $\rho$ and $K$.
\end{proposition}

\begin{remark}
Observe that the assumptions in Theorem \ref{theo:conv-k-rp-spa} (or Definition \ref{defi:strato-sol-spa}) and in Proposition~\ref{prop:asymp-renorm-cstt-spa} cover the case where $d=2$ and $H_1=H_2=\frac12$ . In other words, these results encompass the situation where $\dw$ is a spatial white noise on $\R^2$.
\end{remark}

\section{Proof of Theorem \ref{theo:conv-k-rp}}\label{sec:stochastic-constructions}

This section is devoted to the proof of Theorem \ref{theo:conv-k-rp}, that is to the construction of the $(\al,K)$-rough path $\widehat{\uw}$ at the basis of the Stratonovich interpretation of the model (along Definition \ref{defi:strato-sol}). 

\smallskip

\textit{Therefore, from now on and for the rest of the section, we fix a mollifier $\rho$, \lastchange{some Hurst indexes $H_0,\bh$,} and a parameter $\al$ such that the assumptions in Theorem \ref{theo:conv-k-rp} are all met.}

\smallskip

We recall that the convenient notation $\cn_{H_0,\bh}$ has been introduced in \eqref{notation:cn-h}, and that we have set $c_{H_0,\bh}:=c_{H_0}c_{\bh}$, where $c_{H_0}$ and $c_{\bh}$ are defined by \eqref{cstt-c-h}. For further reference, let us label the following covariance formulas, which immediately generalize \eqref{eq:rep-gamma-fourier} in the regularized setting.

\begin{lemma}
Let $\dw^n$ be the smoothed noise defined by \eqref{def:approximate noise} and recall that the kernel $K$ is defined by~\eqref{decompo-k}. For every fixed $n\geq 1$, the families $\{\dw^n(t,y); (t,y)\in\R^{d+1}\}$ and $\{K\ast\dw^n(t,y); (t,y)\in\R^{d+1}\}$ are centered Gaussian processes with respective covariance functions given by the formulas
\begin{equation}\label{cova-xi-n-xi-n}
\mathbb{E}\big[ \dw^n(t,y) \dw^n(\tti,\yti) \big]=c_{H_0,\bh}^2 \int_{\R^{d+1}} d\la d\xi\, |\cf\rho_n(\la,\xi)|^2 \cn_{H_0,\bh}(\la,\xi)e^{\imath (\la (t-\tti)+  \xi \cdot (y-\yti))}\, ,
\end{equation}
and
\begin{align}
&\mathbb{E}\big[ (K\ast\dw^n)(t,y) (K\ast\dw^n)(\tti,\yti) \big]\nonumber\\
&=c_{H_0,\bh}^2 \int_{\R^{d+1}} d\la d\xi\, |\cf\rho_n(\la,\xi)|^2 |\cf K(\la,\xi)|^2\cn_{H_0,\bh}(\la,\xi)\lastchange{e^{\imath (\la (t-\tti)+  \xi \cdot (y-\yti))}}\, .\label{cova-i-n-i-n}
\end{align}
\end{lemma}

Just as in \cite[Corollary 3.5]{De16}, the proof of Theorem \ref{theo:conv-k-rp} essentially relies on suitable moments estimates (see Proposition \ref{prop:estim-mom-first} and Proposition \ref{prop:estim-mom-second} below). The transition from these estimates to the desired convergence property will then go through the following multiparametric and distributional version of the Garsia-Rodemich-Rumsey Lemma. Observe that this kind of property is one of the key technical ingredients in the theory of regularity structures.

\begin{lemma}[Multiparametric G-R-R lemma]\label{GRR-gene}
Fix a regularity parameter $\beta$ sitting in $(-(d+1),0)$, as well as a weight $w$ on $\R^d$. Then there exists a finite set $\Psi$ of functions in $\cac^{2(d+1)}(\R^{d+1})$ with support in $\cb_\scal(0,1)$ such that the following property holds true: assume that  $\zeta: \R^{d+1} \to \cd'_{2(d+1)}(\R^{d+1})$ is a map with increments of the form
$$\zeta_{s,x}-\zeta_{t,y}=\sum_{i=1}^{r} 
[\theta^i(s,x)-\theta^i(t,y)] \cdot \zeta^{\sharp,i}_{t,y},$$
for some $\theta^i \in \cac_w^\mu(\R^{d+1})$ with $\mu\in [0,\min(1,-\beta))$, and some distributions  $\zeta^{\sharp,i}\in \pmb{\cac}^\beta_w(\R^{d+1})$, where we recall that the spaces $\mathcal{C}^\beta_w$ are introduced in Definition \ref{def:besov-space}. Then for every $T>0$, one has
\begin{align}
&\|\zeta\|_{\beta+\mu;T,w^2}\nonumber\\
&\lesssim \sup_{\psi\in \Psi} \sup_{n\geq 0} \sup_{(s,x)\in \Lambda_\scal^n \cap([-(T+2),T+2]\times \R^d) } 2^{n(\beta+\mu)} \frac{|\langle \zeta_{s,x}, \psi_{s,x}^n \rangle |}{w(x)^2}+\sum_{i=1,\ldots,r} \|\theta^i\|_{\mu;T+2,w}\|\zeta^{\sharp,i}\|_{\beta;T+2,w} \ ,\label{bound-GRR-gene}
\end{align}
where  the discrete set $\Lambda_\scal^n$ is defined by $\Lambda_\scal^n:=\{(2^{-2n} k_0,2^{-n}k_1,\ldots,2^{-n}k_d); \ k_0,k_1,\ldots,k_d\in \Z\}$, and where norms for $\theta^i$ and $\zeta^{\sharp,i}$ are respectively given in Definition \ref{def:besov-space-positive} and \ref{def:besov-space}. For the sake of clarity, we have also used the standard notation $\psi_{s,x}^n:=\mathcal{S}^{2^{-n}}_{s,x}\psi$ in the \lastchange{right-hand side} of ~\eqref{bound-GRR-gene}.
\end{lemma}

\begin{proof}
It is a mere \enquote{weighted} adaptation of the arguments of the proof of \cite[Lemma 3.2]{De16} (which was itself an adaptation of the arguments in \cite[Section 3]{hai-14}). For the sake of conciseness, we leave the details behind this slight adaptation as an exercise to the reader.
\end{proof}

As a last preliminary step, we also label the following elementary property for further use:

\begin{lemma}\label{lem:fouri-psi-prop} Recall that the sets $\cb_\scal^l$ are given by \eqref{cal B^l_s}. Let $\psi$ be a generic element of $\cb_\scal^{2(d+1)}$ and for all \lastchange{$H_0\in (0,1)$, $\mathbf{H}\in (0,1)^{d}$, consider the function $\cn_{H_0,\mathbf{H}}$} introduced in \eqref{notation:cn-h}. Then it holds that
\begin{equation}\label{int-psi-regu}
\int_{\R^{d+1}} d\la d\xi\, \cn_{H_0,\mathbf{H}}(\la,\xi)\big| \cf\psi(\la,\xi) \big| \ < \ \infty \, .
\end{equation}
\end{lemma}

In the above lemma, note that our choice of $\psi \in \cb_\scal^{2(d+1)}$ guarantees strong integrability properties for $\cf\psi$, which are the keys to show that the integral in \eqref{int-psi-regu} is indeed finite.

\subsection{Moment estimate for the first component} In this section we will bound the covariance of $\dw^n$ considered as an element of a space of the form $\mathcal{C}^\alpha$, where $\alpha$ satisfies~\eqref{in th: condition alpha}.

\begin{proposition}\label{prop:estim-mom-first}
For all $\ell\geq 0$, $n\geq m\geq 0$, $\psi\in \cb^{2(d+1)}_\scal$ and $(s,x)\in \R^{d+1}$, it holds that
\begin{align}\label{eq:dw^n-dw^m}\mathbb{E}\big[ |\langle \dw^{n}-\dw^m,\psi^\ell_{s,x} \rangle |^{2}\big] \lesssim 2^{2\ell (d+2-(2H_0+H)+\varepsilon)} 2^{-m\varepsilon}\, ,\end{align}
where the proportional constant in $\lesssim$ does not depend on $n,m,\ell,s,x$ and where we recall that we have set $\psi_{s,x}^{\ell}:=\mathcal{S}^{2^{-\ell}}_{s,x}\psi$. 
\end{proposition}

\begin{proof}
We have by definition
\begin{align*}
&\mathbb{E}\big[ \langle \dw^{n},\psi^\ell_{s,x} \rangle^{2}\big]
=\int_{\R^{d+1}\times \R^{d+1}} dt dy d\tti d\yti \, \psi^\ell_{s,x}(t,y) \psi^\ell_{s,x}(\tti,\yti)\E\big[\dw^n(t,y) \dw^n(\tilde{t},\tilde{y})\big].
\end{align*}
Therefore using the covariance formula \eqref{cova-xi-n-xi-n} together with the definition \eqref{def:Fourier transform} of Fourier transform, we get
\begin{align}
\mathbb{E}\big[ \langle \dw^{n},\psi^\ell_{s,x} \rangle^{2}\big]
&=c_{H_0,\bh}^2 \int_{\R^{d+1}\times \R^{d+1}} dt dy d\tti d\yti 
\, \psi^\ell_{s,x}(t,y) \psi^\ell_{s,x}(\tti,\yti)\nonumber \\
&\hspace{1.5in}\times\int_{\R^{d+1}} d\la d\xi\, |\cf \rho_n(\la,\xi)|^2 \cn_{H_0,\bh}(\la,\xi) 
e^{\imath (\la (t-\tti)+  \xi \cdot (y-\yti))}\nonumber\\
&=c_{H_0,\bh}^2 \int_{\R^{d+1}} d\la d\xi\, |\cf \rho_n(\la,\xi)|^2 \big| \cf\psi^\ell_{s,x}(\la,\xi)\big|^2\cn_{H_0,\bh}(\la,\xi).\label{eq:test W^n1}
\end{align}
We now recall that $\rho_n$ is a rescaled version of the mollifier given by \eqref{def:approximate noise}, and we have also set $\psi^n_{s,x}=\mathcal{S}_{s,x}^{2^{-n}}\psi$ in the \lastchange{right-hand side} of \eqref{bound-GRR-gene}. Hence we obtain
\begin{align}
&\mathbb{E}\big[ \langle \dw^{n},\psi^\ell_{s,x} \rangle^{2}\big]\nonumber\\
&=c_{H_0,\bh}^2 \int_{\R^{d+1}} d\la d\xi\, |\cf\rho(2^{-2n}\la,2^{-n}\xi)|^2 \big| \cf \psi(2^{-2\ell}\la,2^{-\ell}\xi)\big|^2\cn_{H_0,\bh}(\la,\xi).\label{eq:test W^n2}
\end{align}
We now perform the elementary change of variables $\lambda:=2^{-2l}\lambda$ and $\xi:=2^{-l}\xi$, which yields
\begin{align}
&\mathbb{E}\big[ \langle \dw^{n},\psi^\ell_{s,x} \rangle^{2}\big]\nonumber\\
&=c_{H_0,\bh}^2  2^{2\ell (d+2-(2H_0+H))}\int_{\R^{d+1}} d\la d\xi\, |\cf\rho(2^{-2(n-\ell)}\la,2^{-(n-\ell)}\xi)|^2 \big| \cf \psi(\la,\xi)\big|^2\cn_{H_0,\bh}(\la,\xi).\label{eq:test W^n3}
\end{align}
Thanks to \eqref{asympt-rho}, applied with $\tau_0=\cdots=\tau_d=0$, the Fourier transform of $\rho$ is uniformly bounded. Hence we end up with
\begin{align}
\mathbb{E}\big[ \langle \dw^{n},\psi^\ell_{s,x} \rangle^{2}\big]
\lesssim  2^{2\ell (d+2-(2H_0+H))}\int_{\R^{d+1}} d\la d\xi\, \big| \cf \psi(\la,\xi)\big|^2\cn_{H_0,\bh}(\la,\xi)  \ .\label{proof-mom-1}
\end{align}
According to Lemma \ref{lem:fouri-psi-prop} the latter integral is finite, which gives our claim \eqref{eq:dw^n-dw^m} for $m=0$. The general case $m\geq 0$ can then be derived along similar estimates, invoking the fact that $\mathcal{F}\rho$ is a Lipschitz function (see Assumption ($\rho$)).
\end{proof}

\subsection{Moment estimate for the second component}

Let us start with two useful estimates on the Fourier transforms of the (fixed) components $(K,R)$ in the decomposition of the heat kernel (see  relation \eqref{regu-psi-illu}).
\begin{lemma}\label{lem:estim-fourier-k} 
Let $K$ be the localized heat kernel of Definition \ref{defi:loc-heat-ker}. 
For all fixed $a_0,a_1,\ldots,a_d \in [0,1]$ such that $\sum_{i=0}^d a_i<1$, one has, for every $(\la,\xi)\in \R^{d+1}$,
$$|\cf K(\la,\xi)|\lesssim |\la|^{-a_0} \prod_{i=1}^d |\xi_{i}|^{-2a_i} \ .$$
\end{lemma}

\begin{proof}
Using the expansion of $K$ in (\ref{decompo-k}) and recalling the definition \eqref{scaling-operator} of $\mathcal{S}_{s,x}^\delta$, we can first write
\begin{align}\label{proof midstep FK}\cf K(\la,\xi)=\sum_{\ell \geq 0} 2^{-2\ell} \cf K_0(2^{-2\ell}\la,2^{-\ell}\xi) \, .\end{align}
Then, since $K_0$ is a smooth compactly-supported function, one has $|\cf K_0(\la,\xi)|\lesssim |\la|^{-\tau_0}$ and $|\cf K_0(\la,\xi)|\lesssim |\xi_i|^{-\tau_i}$ for all $\tau_0,\tau_1,\ldots,\tau_d \geq 0$ and $(\la,\xi)\in \R^{d+1}$. Plugging this information into~\eqref{proof midstep FK}, we get
\begin{align*}
\big| \cf K(\la,\xi) \big| &\leq \sum_{\ell \geq 0} 2^{-2\ell} \big|\cf K_0(2^{-2\ell}\la,2^{-\ell}\xi)\big|^{a_0}\cdots \big|\cf K_0(2^{-2\ell}\la,2^{-\ell}\xi)\big|^{a_d}\\
&\lesssim |\la|^{-a_0} \prod_{i=1}^d |\xi_i|^{-2a_i} \sum_{\ell \geq 0} 2^{-2\ell (1-(a_0+a_1+\cdots+a_d))} \lesssim |\la|^{-a_0} \prod_{i=1}^d |\xi_i|^{-2a_i}  \, ,
\end{align*}
which finishes our proof.
\end{proof}

We now turn to a bound concerning the function $R$ involved in the decomposition \eqref{decompo-k}.
\begin{lemma}\label{lem:estim-fouri-r}
Let $R$ be the remainder term associated with the localized heat kernel $K$ (along Definition \ref{defi:loc-heat-ker}). Then, for all fixed $a_0,a_1,\ldots,a_d \geq 0$ such that $\sum_{i=0}^d a_i> 1$, one has, for every $(\la,\xi)\in \R^{d+1}$,
\begin{equation}\label{estim-fouri-r}
|\cf R(\la,\xi)|\lesssim |\la|^{-a_0} \prod_{i=1}^d |\xi_i|^{-2a_i} \ .
\end{equation}
As a consequence, if $H_0\in (0,1),\bh=(H_1,\ldots,H_d)\in (0,1)^{d}$ are such that $2H_0+H < d+1$, the following relation holds true for the function $\cn_{H_0,\bh}$ defined by \eqref{notation:cn-h}:
\begin{equation}\label{integr-r}
\int_{\R^{d+1}} d\la d\xi\, \cn_{H_0,\bh}(\la,\xi) \big| \cf R(\la,\xi)\big| < \infty .
\end{equation}
\end{lemma}

\begin{proof}
Using the expansion of $R$ in \eqref{decompo-k} and relation \eqref{scaling-operator} for $\mathcal{S}_{s,x}^\delta$, we can first write
$$\cf R(\la,\xi)=\sum_{\ell >0} 2^{2\ell} \cf K_0(2^{2\ell}\la,2^\ell\xi) \, .$$
Then, similarly to what we did in the proof of Lemma \ref{lem:estim-fourier-k}, we invoke the bound $|\cf K_0(\la,\xi)|\lesssim |\la|^{-\tau_0}$ and $|\cf K_0(\la,\xi)|\lesssim |\xi_i|^{-\tau_i}$ for all $\tau_0,\tau_1,\ldots,\tau_d \geq 0$ and $(\la,\xi)\in \R^{d+1}$.  We deduce that for any $a_0,\dots,a_d\geq0$ such that $\sum_{i=0}^d a_i> 1$ we have
\begin{align*}
\big| \cf R(\la,\xi) \big| &\leq \sum_{\ell >0} 2^{2\ell} \big|\cf K_0(2^{2\ell}\la,2^\ell\xi)\big|^{1/(d+1)}\cdots \big|\cf K_0(2^{2\ell}\la,2^\ell\xi)\big|^{1/(d+1)}\\
&\lesssim |\la|^{-a_0} \prod_{i=1}^d |\xi_i|^{-2a_i} \sum_{\ell >0} 2^{2\ell (1-(a_0+a_1+\cdots+a_d))} \lesssim |\la|^{-a_0} \prod_{i=1}^d |\xi_i|^{-2a_i}  \, .
\end{align*}
This proves the assertion \eqref{estim-fouri-r}.

We now turn to a bound on the integral introduced in \eqref{integr-r}. To this aim, we split the integral according to the region $\cd_\scal$ defined below by  \eqref{def: calD_s} and  we recall that $R=p-K$, which yields
\begin{multline}
\int_{\R^{d+1}} d\la d\xi\, \cn_{H_0,\bh}(\la,\xi) \big| \cf R(\la,\xi)\big|
\lesssim \bigg[\int_{\cd_\scal} d\la d\xi\, \cn_{H_0,\bh}(\la,\xi)\big| \cf p(\la,\xi)\big| \\
+\int_{\cd_\scal} d\la d\xi\, \cn_{H_0,\bh}(\la,\xi) \big| \cf K(\la,\xi)\big|\bigg]
+\int_{\R^{d+1}\backslash \cd_\scal} d\la d\xi\, \cn_{H_0,\bh}(\la,\xi) \big| \cf R(\la,\xi)\big|\, .\label{boun-aur-detail}
\end{multline}
Next, taking into account expression \eqref{eq:fourier-p} for the Fourier transform of $p$, the integral 
$$\int_{\cd_\scal} d\la d\xi\, \cn_{H_0,\bh}(\la,\xi) | \cf p(\la,\xi)|$$
in~\eqref{boun-aur-detail} is (essentially) the same as in the right-hand side of~\eqref{eq:cj-c-n-border}. We have already shown that this integral is finite in the proof of Lemma \ref{lem:cstt-renorm}. In addition, one can bound $| \cf K(\la,\xi)|$ by a constant thanks to Lemma~\ref{lem:estim-fourier-k}, in order to get
\begin{equation*}
\int_{\cd_\scal}  \cn_{H_0,\bh}(\la,\xi) \big| \cf K(\la,\xi) \big|\, d\la d\xi
\lesssim
\int_{\cd_\scal} d\la d\xi\, \cn_{H_0,\bh}(\la,\xi) 
<\infty .
\end{equation*}
Eventually, the finiteness of $\int_{\R^{d+1}\backslash \cd_\scal} d\la d\xi\, \cn_{H_0,\bh}(\la,\xi) | \cf R(\la,\xi)|$ can be easily derived from relation~\eqref{estim-fouri-r}. Plugging the information  above into~\eqref{boun-aur-detail}, this completes the proof of our claim~\eqref{integr-r}.
\end{proof}

As we will see in the sequel, the renormalization procedure for $\bw^{\mathbf{2},n}$ is based on the following decomposition.

\begin{lemma}\label{lem:decompo-esp-xi-2} 
Let $\bw^{\mathbf{2},n}$ be the increment given by~\eqref{notation-xi-2}, and recall that the renormalization constant $\mathfrak{c}^{(n)}_{\rho,H_0,\bh}$ is defined by \eqref{renormal}. Then for all $(s,x),(t,y)\in \R^{d+1}$ and $n\geq 1$, one has the decomposition
\begin{align}\label{decomp EW^2n_sx}\mathbb{E}\big[ \bw^{\mathbf{2},n}_{s,x}(t,y)\big]=\mathfrak{c}^{(n)}_{\rho,H_0,\bh}+\ce^n_{s,x}(t,y) \, ,\end{align}
for some function $\ce^n_{s,x}$ such that, for every $\varepsilon\in (0,1)$, $\ell\ge 0$ and  $\psi\in\cb_\scal^\ell$ we have
\begin{equation}\label{estim-mathcal-e}
\big|\langle \ce^n_{s,x},\psi^\ell_{s,x}\rangle \big| \lesssim  2^{2\ell (1+d-(2H_0+H)+\varepsilon)} \, .
\end{equation}
Moreover, in relation \eqref{estim-mathcal-e} the proportional constant does not depend on $n,\ell,s,x$.
\end{lemma}

\begin{proof}With the definition \eqref{notation-xi-2} of $\bw^{\mathbf{2},n}$ in mind, we can obviously write 
$$\mathbb{E}\big[ \bw^{\mathbf{2},n}_{s,x}(t,y)\big]=\mathfrak{c}^{(n)}_{\rho,H_0,\bh}+\ce^n_{s,x}(t,y),
$$ 
as stated in \eqref{decomp EW^2n_sx},  where we have simply set
\begin{align}\label{eq:Epsi^n_sx}
\ce^n_{s,x}(t,y):=\Big\{\mathbb{E}\big[(K\ast \dw^n)(t,y) \dw^n(t,y) \big]-\mathfrak{c}^{(n)}_{\rho,H_0,\bh} \Big\}-\mathbb{E}\big[(K\ast \dw^n)(s,x) \dw^n(t,y) \big] \, .
\end{align}
We now analyze the terms
\begin{align}\label{eq:Q^n(s,x;t,y)}
Q^n(s,x;t,y)=\mathbb{E}\big[(K\ast \dw^n)(s,x) \dw^n(t,y) \big]
\end{align}
in the \lastchange{right-hand side} of \eqref{eq:Epsi^n_sx}. To this aim, we resort to  a slight variation on~\eqref{cova-xi-n-xi-n} and~\eqref{cova-i-n-i-n}, which enables to write that for all $(s,x),(t,y)\in \R^{d+1}$ 
\begin{align*}
&Q^n(s,x;t,y)
=c_{H_0,\bh}^2 \int_{\R^{d+1}}  |\cf\rho_n(\la,\xi)|^2 \cn_{H_0,\bh}(\la,\xi) \cf K(\la,\xi)e^{\imath (\la (t-s)+  \xi \cdot (y-x))} \, d\la d\xi\,.
\end{align*}

Based on this expression, and along the same lines as for \eqref{eq:test W^n1}, one gets on the one hand
\begin{align*}
&\int_{\R^{d+1}} dtdy\, Q^n(s,x;t,y)\psi^\ell_{s,x}(t,y) \\
&=c_{H_0,\bh}^2 \int_{\R^{d+1}} d\la d\xi\, |\cf\rho_n(\la,\xi)|^2 \cn_{H_0,\bh}(\la,\xi) \cf K(\la,\xi)\cf \psi^{\ell}_{0,0}(\lambda,\xi).\end{align*}
Hence owing to the fact that $\psi^\ell_{0,0}=\mathcal{S}^{2^{-\ell}}_{0,0}\psi$ and performing the change of variable $\lambda:=2^{-2\ell}\lambda, \xi=2^{-\ell}\xi$, we get
\begin{align*}
&\bigg| \int_{\R^{d+1}} dtdy\, Q^n(s,x;t,y)\psi^\ell_{s,x}(t,y) \bigg| \\
&=c_{H_0,\bh}^2\, 2^{2\ell (d+2-(2H_0+H))} \bigg|\int_{\R^{d+1}} d\la d\xi\, |\cf\rho_n(2^{2\ell}\la,2^\ell\xi)|^2 \cn_{H_0,\bh}(\la,\xi)\cf K(2^{2\ell}\la,2^\ell\xi)\cf\psi(\la,\xi) \bigg| \, .
\end{align*}
At this point, observe that due to the assumption $2H_0+H \leq d+1$, we can pick $a_0,a_1,\ldots,a_d$ in $[0,1]$ such that $\sum_{i=0}^d a_i=1-\varepsilon$, $2H_0+a_0-1<1$ and $2H_i+2a_i-1<1$ for $i=1,\ldots,d$. We can now apply Lemma \ref{lem:estim-fourier-k} with this set of parameters to deduce that
\begin{align}
&\bigg| \int_{\R^{d+1}} dtdy\, Q^n(s,x;t,y)\psi^\ell_{s,x}(t,y) \bigg|\nonumber \\
&\lesssim2^{2\ell (d+1-(2H_0+H)+\varepsilon)} \int_{\R^{d+1}} d\la d\xi\, \frac{1}{|\la|^{2H_0+a_0-1}} \prod_{i=1}^d \frac{1}{|\xi_i|^{2H_i+2a_i-1}} \big| \cf\psi(\la,\xi) \big| \, .\label{first-bo}
\end{align}
Since $2H_0+a_0<2$ and $2H_i+2a_i<2$ for $i=1,\ldots,d$, we can finally appeal to Lemma \ref{lem:fouri-psi-prop} to assert that the latter integral is finite, which gives the desired bound for the second term in the right-hand side of \eqref{eq:Epsi^n_sx}.

\smallskip

Then, for the treatment of the difference into brackets in \eqref{eq:Epsi^n_sx}, let us separate the two cases $2H_0+H<d+1$ and $2H_0+H=d+1$.

\smallskip

\noindent
\underline{First case: $2H_0+H<d+1$.} In this situation, going back to the definition \eqref{eq:cj-c-n} of $\cj_{\rho,H_0,\bh}$, observe that the renormalization constant can also be expressed as
\begin{equation*}
\mathfrak{c}^{(n)}_{\rho,H_0,\bh}=c_{H_0,\bh}^2 \int_{\R^{d+1}}  |\cf\rho_n(\la,\xi)|^2  \cn_{H_0,\bh}(\la,\xi)\cf p(\la,\xi)  \, d\la d\xi\, ,
\end{equation*}
and accordingly
\begin{align*}
&Q^n(t,y;t,y)-\mathfrak{c}^{(n)}_{\rho,H_0,\bh}=-c_{H_0,\bh}^2 \int_{\R^{d+1}}  |\cf\rho_n(\la,\xi)|^2 \cn_{H_0,\bh}(\la,\xi) \cf R(\la,\xi) \, d\la d\xi\,,
\end{align*}
where $R$ stands for the remainder term in the decomposition of Definition \ref{defi:loc-heat-ker}, item $(i)$. Invoking the inequality $|\cf\rho_n(\la,\xi)| \lesssim 1$ and the result of \eqref{integr-r}, we get
\begin{equation}\label{second-bo}
\Big|Q^n(t,y;t,y)-\mathfrak{c}^{(n)}_{\rho,H_0,\bh}\Big| \lesssim 1 \le 2^{2\ell (1+d-(2H_0+H)+\varepsilon)} \, ,
\end{equation}
where the last inequality naturally stems from the fact that $2H_0+H<d+1$.

\smallskip

\noindent
\underline{Second case: $2H_0+H=d+1$.} Let us recall that in this situation,
\begin{equation*}
\mathfrak{c}^{(n)}_{\rho,H_0,\bh}=c_{H_0,\bh}^2 \int_{|\la|+|\xi|^2\geq 2^{-2n}}  |\mathcal F{\rho}(\la,\xi)|^2  \mathcal F{p}(\la,\xi) \cn_{H_0,\bh}(\la,\xi)\, d\la d\xi \, .
\end{equation*}
In fact, using the relation $2H_0+H=d+1$, it is not hard to check that we can recast the above quantity as
\begin{align*}
&\mathfrak{c}^{(n)}_{\rho,H_0,\bh}=c_{H_0,\bh}^2 \int_{|\la|+|\xi|^2\geq 1}  |\mathcal F{\rho_n}(\la,\xi)|^2  \mathcal F{p}(\la,\xi) \cn_{H_0,\bh}(\la,\xi)\, d\la d\xi \, ,
\end{align*}
and accordingly
\begin{align*}
&Q^n(t,y;t,y)-\mathfrak{c}^{(n)}_{\rho,H_0,\bh}=c_{H_0,\bh}^2 \int_{|\la|+|\xi|^2\leq 1}  |\cf\rho_n(\la,\xi)|^2 \cn_{H_0,\bh}(\la,\xi) \cf K(\la,\xi) \, d\la d\xi\\
&\hspace{5cm}-c_{H_0,\bh}^2 \int_{|\la|+|\xi|^2\geq 1}  |\cf\rho_n(\la,\xi)|^2 \cn_{H_0,\bh}(\la,\xi) \cf R(\la,\xi) \, d\la d\xi\,.
\end{align*}
Using the results of Lemma \ref{lem:estim-fourier-k} and Lemma \ref{lem:estim-fouri-r}, as well as the uniform estimate $|\cf\rho_n(\la,\xi)| \lesssim 1$, we thus get
\begin{align}
& \Big|Q^n(t,y;t,y)-\mathfrak{c}^{(n)}_{\rho,H_0,\bh}\Big|\nonumber\\
&\lesssim \int_{|\la|+|\xi|^2\leq 1}  \cn_{H_0,\bh}(\la,\xi) \, d\la d\xi+ \int_{|\la|+|\xi|^2\geq 1}   \cn_{H_0,\bh}(\la,\xi) |\cf R(\la,\xi)| \, d\la d\xi  \ \lesssim \ 1
\ \le \ 2^{2\ell \varepsilon} \, ,\label{third-bo}
\end{align}
which corresponds to the desired bound in this case.

We can now conclude our proof: combining \eqref{first-bo}, \eqref{second-bo} and \eqref{third-bo} with \eqref{eq:Epsi^n_sx}, we immediately obtain \eqref{estim-mathcal-e}.

\end{proof}

We turn to a bound on the variance of the renormalized $K$-rough path $\widehat{\bw}^n$.
\begin{proposition}\label{prop:estim-mom-second} Let $\widehat{\bw}^n$ be the renormalized $K$-rough path defined by \eqref{hat W^n}, where we recall that $\pmb{W}^{n}:=(\dw^{n},\mathbf{W}^{\mathbf{2},n})$ and $\mathbf{W}^{\mathbf{2},n}$ is introduced in~\eqref{notation-xi-2}. Then for all $\ell\geq 0$, $n\ge m \ge 0$, $\psi\in \cb^{2(d+1)}_\scal$, $(s,x)\in \R^{d+1}$ and $\varepsilon \in (0,1)$, it holds that
\begin{align}\label{eq:hatW^n-hatW^m}\mathbb{E}\big[ |\langle \widehat{\bw}^{\mathbf{2},n}_{s,x}-\widehat{\bw}^{\mathbf{2},m}_{s,x},\psi^\ell_{s,x} \rangle |^{2}\big] \lesssim 2^{4\ell (1+d-(2H_0+H)+\varepsilon)}2^{-m\varepsilon} \, ,\end{align}
where the proportional constant in \eqref{eq:hatW^n-hatW^m} does not depend on $n,m,\ell,s,x$.
\end{proposition}

\begin{proof}

For the sake of conciseness, we will only focus on the case $m=0$, i.e. we will show the uniform estimate
$$\mathbb{E}\big[ |\langle \widehat{\bw}^{\mathbf{2},n}_{s,x},\psi^\ell_{s,x} \rangle |^{2}\big] \lesssim 2^{4\ell (1+d-(2H_0+H)+\varepsilon)} \, .$$
The proof in the general case $m\geq 0$ could in fact be obtained through elementary adaptations of the subsequent estimates, using the fact that $\cf \rho$ is Lipschitz (see e.g. the arguments of~\cite{De16} for more details on the transition from $m=0$ to $m\geq 0$).

Observe first that due to Wick's formula for products of Gaussian random variables (and using the notation of \eqref{notation-xi-2}), we can write
\begin{align*}
\mathbb{E}\big[|\langle \bw^{\mathbf{2},n}_{s,x},\psi^\ell_{s,x} \rangle |^{2}\big]&=\iint_{\R^{d+1}\times \R^{d+1}} dt dy d\tti d\yti \, \psi^\ell_{s,x}(t,y) \psi^\ell_{s,x}(\tti,\yti)\mathbb{E}\big[\ci^n_{s,x}(t,y) \dw^{n}(t,y)\ci^n_{s,x}(\tti,\yti) \dw^{n}(\tti,\yti) \big]\\
&=\big(\big\langle \mathbb{E}\big[\bw^{\mathbf{2},n}_{s,x}\big],\psi^\ell_{s,x} \big\rangle  \big)^2+\mathcal{U}^{\ell,n}_{s,x}+\mathcal{V}^{\ell,n}_{s,x} \, ,
\end{align*}
where we have set
$$\mathcal{U}^{\ell,n}_{s,x}:=\iint_{\R^{d+1}\times \R^{d+1}} dt dy d\tti d\yti \, \psi^\ell_{s,x}(t,y) \psi^\ell_{s,x}(\tti,\yti)\mathbb{E}\big[\ci^n_{s,x}(t,y) \ci^n_{s,x}(\tti,\yti)  \big]\mathbb{E}\big[ \dw^{n}(t,y)\dw^{n}(\tti,\yti) \big]$$
and
$$\mathcal{V}^{\ell,n}_{s,x}:=\iint_{\R^{d+1}\times \R^{d+1}} dt dy d\tti d\yti \, \psi^\ell_{s,x}(t,y) \psi^\ell_{s,x}(\tti,\yti)\mathbb{E}\big[\ci^n_{s,x}(t,y)  \dw^{n}(\tti,\yti) \big]\mathbb{E}\big[\dw^{n}(t,y)\ci^n_{s,x}(\tti,\yti) \big] \ .$$
Based on this decomposition, we get that
\begin{align*}
&\mathbb{E}\big[ |\langle \widehat{\bw}^{\mathbf{2},n}_{s,x},\psi^\ell_{s,x} \rangle |^{2}\big]=\mathbb{E}\big[ |\langle \bw^{\mathbf{2},n}_{s,x}-\mathfrak{c}^{(n)}_{\rho,H_0,\bh},\psi^\ell_{s,x} \rangle |^{2}\big]\\
&=\big(\big\langle \mathbb{E}\big[\bw^{\mathbf{2},n}_{s,x}\big],\psi^\ell_{s,x} \big\rangle  \big)^2+\mathcal{U}^{\ell,n}_{s,x}+\mathcal{V}^{\ell,n}_{s,x} -2 \langle \mathbb{E}\big[\bw^{\mathbf{2},n}_{s,x}\big],\psi^\ell_{s,x} \big\rangle \langle \mathfrak{c}^{(n)}_{\rho,H_0,\bh},\psi^\ell_{s,x}\rangle+\langle \mathfrak{c}^{(n)}_{\rho,H_0,\bh},\psi^\ell_{s,x}\rangle^2\\
&=\big(\big\langle \mathbb{E}\big[\bw^{\mathbf{2},n}_{s,x}\big]-\mathfrak{c}^{(n)}_{\rho,H_0,\bh},\psi^\ell_{s,x} \big\rangle  \big)^2+\mathcal{U}^{\ell,n}_{s,x}+\mathcal{V}^{\ell,n}_{s,x}\\
&=\big(\big\langle \mathcal{E}^n_{s,x},\psi^\ell_{s,x} \big\rangle  \big)^2+\mathcal{U}^{\ell,n}_{s,x}+\mathcal{V}^{\ell,n}_{s,x} \, ,
\end{align*}
where we have used Lemma \ref{lem:decompo-esp-xi-2} (and the notation therein) to derive the last identity.
Owing to~\eqref{estim-mathcal-e}, our claim \eqref{eq:hatW^n-hatW^m} is thus reduced to check that
\begin{align}\label{eq:bound on U and V}
\big|\mathcal{U}^{\ell,n}_{s,x} \big| \lesssim 2^{4\ell (1+d-(2H_0+H)+\varepsilon)}\quad \textrm{and}\quad  \big|\mathcal{V}^{\ell,n}_{s,x} \big| \lesssim 2^{4\ell (1+d-(2H_0+H)+\varepsilon)}.
\end{align} 
The remainder of the proof is devoted to prove \eqref{eq:bound on U and V}.

\smallskip

To this end, recall that $\ci^n_{s,x}$ is defined by \eqref{notationci-n}, which, together with relation \eqref{cova-i-n-i-n}, yields
\begin{align*}
&\mathbb{E}\big[\ci^n_{s,x}(t,y) \ci^n_{s,x}(\tti,\yti)  \big]=c_{H_0,\bh}^2 \int_{\R^{d+1}} d\la d\xi\, |\cf\rho_n(\la,\xi)|^2 |\cf K(\la,\xi)|^2\cn_{H_0,\bh}(\la,\xi) \\
&\hspace{6cm}\big[ e^{\imath (\la (t-\tti)+  \xi \cdot (y-\yti))}-e^{\imath (\la (t-s)+  \xi \cdot (y-x))}-e^{\imath (\la (s-\tti)+  \xi \cdot (x-\yti))}+ 1\big]
\end{align*}
Combining this expression with formula \eqref{cova-xi-n-xi-n} for $\mathbb{E}\big[ \dw^{n}(t,y)\dw^{n}(\tti,\yti) \big]$, we easily deduce that 
\begin{align}
&\mathcal{U}^{\ell,n}_{s,x}=c_{H_0,\bh}^2 \iint_{\R^{d+1}\times \R^{d+1}} d\la d\xi d\lati d\xiti \,  |\cf\rho_n(\la,\xi)|^2 |\cf\rho_n(\lati,\xiti)|^2|\cf K(\la,\xi)|^2 \cn_{H_0,\bh}(\la,\xi) \cn_{H_0,\bh}(\lati,\xiti)\nonumber\\
&\hspace{3cm}\Big[\big|\cf\psi^\ell_{s,x}(\la+\lati,\xi+\xiti)\big|^2 -\overline{\cf\psi^\ell_{s,x}(\la+\lati,\xi+\xiti)}\cf\psi_{s,x}^\ell(\lati,\xiti) e^{-\imath(\la s+\xi\cdot x)}\nonumber\\
&\hspace{5cm}-\cf\psi_{s,x}^\ell(\la+\lati,\xi+\xiti) \overline{\cf\psi^\ell_{s,x}(\lati,\xiti)}e^{\imath(\la s+\xi\cdot x)}+\big|\cf\psi^\ell_{s,x}(\la,\xi)\big|^2\Big]\nonumber\\
&=c_{H_0,\bh}^2 \iint_{\R^{d+1}\times \R^{d+1}} d\la d\xi d\lati d\xiti \,  |\cf\rho_n(\la,\xi)|^2 |\cf\rho_n(\lati,\xiti)|^2|\cf K(\la,\xi)|^2 \cn_{H_0,\bh}(\la,\xi) \cn_{H_0,\bh}(\lati,\xiti)\nonumber\\
&\hspace{3cm}\Big[\big|\cf\psi^\ell_{0,0}(\la+\lati,\xi+\xiti)\big|^2 -\overline{\cf\psi^\ell_{0,0}(\la+\lati,\xi+\xiti)}\cf\psi_{0,0}^\ell(\lati,\xiti) \nonumber\\
&\hspace{6cm}-\cf\psi_{0,0}^\ell(\la+\lati,\xi+\xiti) \overline{\cf\psi^\ell_{0,0}(\lati,\xiti)}+\big|\cf\psi^\ell_{0,0}(\la,\xi)\big|^2\Big]\nonumber\\
&=c_{H_0,\bh}^2 \iint_{\R^{d+1}\times \R^{d+1}} d\la d\xi d\lati d\xiti \,  |\cf\rho_n(\la,\xi)|^2 |\cf\rho_n(\lati,\xiti)|^2|\cf K(\la,\xi)|^2 \cn_{H_0,\bh}(\la,\xi) \cn_{H_0,\bh}(\lati,\xiti)\nonumber\\
&\hspace{8cm}\big| \cf\psi^\ell_{0,0}(\la+\lati,\xi+\xiti)-\cf\psi^\ell_{0,0}(\lati,\xiti)\big|^2 \, .\label{boun-u-ell-det}
\end{align}
Along similar arguments, we obtain first
\begin{align*}
&\mathcal{V}^{\ell,n}_{s,x}=c_{H_0,\bh}^2\\
& \iint_{\R^{d+1}\times \R^{d+1}} d\la d\xi d\lati d\xiti \,  |\cf\rho_n(\la,\xi)|^2 |\cf\rho_n(\lati,\xiti)|^2\cf K(\la,\xi)\overline{\cf K(\lati,\xiti)} \cn_{H_0,\bh}(\la,\xi) \cn_{H_0,\bh}(\lati,\xiti)\\
&\hspace{3cm}\overline{\big[ \cf\psi^\ell_{0,0}(\la+\lati,\xi+\xiti)-\cf\psi^\ell_{0,0}(\la,\xi)\big]}\big[ \cf\psi^\ell_{0,0}(\la+\lati,\xi+\xiti)-\cf\psi^\ell_{0,0}(\lati,\xiti)\big] \, ,
\end{align*}
and we can now apply Cauchy-Schwarz inequality to derive the estimate
\begin{align}
&\big|\mathcal{V}^{\ell,n}_{s,x}\big|\leq c_{H_0,\bh}^2\iint_{\R^{d+1}\times \R^{d+1}} d\la d\xi d\lati d\xiti \,  |\cf\rho_n(\la,\xi)|^2 |\cf\rho_n(\lati,\xiti)|^2|\cf K(\la,\xi)|^2 \cn_{H_0,\bh}(\la,\xi) \cn_{H_0,\bh}(\lati,\xiti)\nonumber\\
&\hspace{8cm}\big| \cf\psi^\ell_{0,0}(\la+\lati,\xi+\xiti)-\cf\psi^\ell_{0,0}(\lati,\xiti)\big|^2 \, .\label{boun-v-ell-det}
\end{align}
Combining \eqref{boun-u-ell-det}-\eqref{boun-v-ell-det} with the uniform bound $|\cf\rho_n(\la,\xi)| \lesssim 1$, we have thus shown that uniformly in $(s,x)\in\R^{d+1}$ and $n\geq1$ the following holds true:
\begin{align}\label{eq:uniform bound of UV by calS}
|\mathcal{U}^{\ell,n}_{s,x}|+|\mathcal{V}^{\ell,n}_{s,x}|\lesssim\cs^{\ell},
\end{align}
where the quantity $\cs^\ell$ is given by 
\begin{align*}
\cs^{\ell}:=c_{H_0,\bh}^4 \iint_{\R^{d+1}\times \R^{d+1}} d\la d\xi d\lati d\xiti \, |\cf K(\la,\xi)|^2 &\cn_{H_0,\bh}(\la,\xi) \cn_{H_0,\bh}(\lati,\xiti) \\ 
&\times\big| \cf\psi^\ell_{0,0}(\la+\lati,\xi+\xiti)-\cf\psi^\ell_{0,0}(\lati,\xiti)\big|^2.
\end{align*}
Moreover, an easy scaling argument performed on $\psi^{\ell}_{0,0}=\cs^{2^\ell}_{0,0}\psi$ shows that
\begin{align*}
&\cs^{\ell}
=c_{H_0,\bh}^4 \, 2^{4\ell(d+2-(2H_0+H))}\tilde{\cs}^\ell,
\end{align*}
where
\begin{align}
\tilde{\cs}^\ell=
\iint_{\R^{d+1}\times \R^{d+1}} d\la d\xi d\lati d\xiti \, |\cf K(2^{2\ell}\la,2^\ell \xi)|^2 &\cn_{H_0,\bh}(\la,\xi) \cn_{H_0,\bh}(\lati,\xiti) \nonumber\\
&\times\big| \cf\psi(\la+\lati,\xi+\xiti)-\cf\psi(\lati,\xiti)\big|^2 \, .\label{cs-ell}
\end{align}
Plugging this information into \eqref{eq:uniform bound of UV by calS} and then \eqref{eq:bound on U and V} we are now reduced to show that for any $\epsilon\in(0,1)$ we have
\begin{align}\label{eq:bound for tilde sc^l}
\tilde{\cs}^\ell\lesssim 2^{-4\ell(1-\varepsilon)}. 
\end{align}
We shall prove assertion \eqref{eq:bound for tilde sc^l} in the next subsection.
\end{proof}

\subsection{Proof of \eqref{eq:bound for tilde sc^l}}

Let us start by highlighting a few inequalities satisfied by $(H_0,\bh)$, that will serve us later in the proof. First, observe that due to \eqref{strengthened 4.17} and $H\le d$, one has $d+\frac12<2H_0+H<2H_0+d$, and so one has necessarily
\begin{equation}\label{impli-h-0}
H_0>\frac14.
\end{equation}
Likewise, it holds that $d+\frac12<2H_0+H<2H_0+H_1+(d-1)$, and so
\begin{equation}\label{impli-h-1}
2H_0+H_1>\frac32,
\end{equation}
while for $d\geq 2$, one has $d+\frac12<2H_0+H_1+H_2+(d-2)$, and so
\begin{equation}\label{impli-h-2}
2H_0+H_1+H_2>\frac52.
\end{equation}
Besides, for obvious symmetry reasons in both expression \eqref{cs-ell} of $\tilde{\cs}^\ell$ and condition \eqref{strengthened 4.17} on $H$, \emph{we can and will assume in the sequel that $H_1\leq H_2 \leq \ldots \leq H_d$}. As a consequence of this assumption, we get that for $d\geq 3$ and $i\geq 3$, $d+\frac12<2H_0+H<2H_0+H_1+H_2+H_3+(d-3)<2+3H_i+(d-3)$, and therefore
\begin{equation}\label{observ-h-i}
H_i >\frac12 \quad \text{for any} \ i \geq 3 \, .
\end{equation}

With these conditions in hand, let us go back to our main purpose, that is proving the estimate~\eqref{eq:bound for tilde sc^l}. With \eqref{cs-ell} in mind, our bound on $\tilde{\cs}^\ell$ relies on a proper control of the difference 
$$\big| \cf\psi(\la+\lati,\xi+\xiti)-\cf\psi(\lati,\xiti)\big|.$$
To this aim, let us introduce some additional notation. Namely for $\lambda,\tilde{\lambda}\in\R$ we set 
\begin{equation}\label{defi-t-0}
\mathcal{T}^{(0)}(\la):=\bigg(\int_{\R^{d+1}} dt dy \, |(\partial_{t x_1 \cdots x_d} \psi)(t,y) | \bigg| \int_0^t du \, e^{-\imath \la u}  \bigg|^{d+1}\bigg)^{1/(d+1)} \, ,
\end{equation}
\begin{equation}\label{defi-q-0}
\mathcal{Q}^{(0)}(\la,\tilde{\la}):=\bigg(\int_{\R^{d+1}} dt dy \, |(\partial_{tx_1\cdots x_d} \psi)(t,y) | \bigg| \int_0^t du\int_0^u dv \, e^{-\imath \tilde{\la} u} e^{-\imath \la v} \bigg|^{d+1}\bigg)^{1/(d+1)} \, ,
\end{equation}
and for $i=1,\ldots,d$,
\begin{equation}\label{defi-t-i}
\mathcal{T}^{(i)}(\la):=\bigg(\int_{\R^{d+1}} dt dy \, |(\partial_{tx_1\cdots x_d} \psi)(t,y) | \bigg| \int_0^{y_i} dz_i \, e^{-\imath \la z_i}  \bigg|^{d+1}\bigg)^{1/(d+1)} \, ,
\end{equation}
\begin{equation}\label{defi-q-i}
\mathcal{Q}^{(i)}(\la,\tilde{\la}):=\bigg(\int_{\R^{d+1}} dt dy \, |(\partial_{tx_1 \cdots x_d} \psi)(t,y) | \bigg| \int_0^{y_i} dz_i\int_0^{z_i} dw_i \, e^{-\imath \tilde{\la} z_i} e^{-\imath \la w_i} \bigg|^{d+1}\bigg)^{1/(d+1)} \, .
\end{equation}
Using this notation, some elementary algebraic manipulations reveal that for all $\lambda,\tilde{\lambda}\in \R$ and $\xi,\tilde{\xi}$ in $\R^{d}$, we have
\begin{equation}\label{bou-var-psi-1-1}
\big| \cf\psi(\la+\lati,\xi+\xiti)-\cf\psi(\lati,\xi+\xiti)\big| 
\lesssim 
|\la| \, \mathcal{Q}^{(0)}(\la,\lati) \,  \prod_{i=1}^{d} \mathcal{T}^{(i)}(\xi_i+\xiti_i) .
\end{equation}
Along the same lines, for $i=1,\ldots,d$ we also get 
\begin{align}
&\big| \cf\psi(\lati,\xiti_1,\ldots,\xiti_{i-1},\xi_i+\xiti_i,\xi_{i+1}+\xiti_{i+1},\ldots,\xi_d+\xiti_d)
\nonumber\\
&\hspace{5cm}
-\cf\psi(\lati,\xiti_1,\ldots,\xiti_{i-1},\xiti_i,\xi_{i+1}+\xiti_{i+1},\ldots,\xi_d+\xiti_d)\big|
\nonumber\\
 &\lesssim 
 \mathcal{T}^{(0)}(\lati) \,
  \bigg(\prod_{j=1}^{i-1} \mathcal{T}^{(j)}(\xiti_j)\bigg) \,
 \left(|\xi_i| \cdot\mathcal{Q}^{(i)}(\xi_i,\xiti_i)\right) \,
 \bigg(\prod_{j=i+1}^{d} \mathcal{T}^{(j)}(\xi_j+\xiti_{j})\bigg)
 \, .\label{bou-var-psi-1-2}
\end{align}
We now point out a lemma on the functions $\mathcal{T}^{(i)}$ and $\mathcal{Q}^{(i)}$ which will be crucial in the sequel. 

\begin{lemma}\label{lem:t-i-q-i-improv}
Fix $\psi\in \cac^{d+1}(\R^{d+1};\R)$ with compact support, $i\in \{0,1,\ldots,d\}$, and let $\mathcal{T}^{(i)},\mathcal{Q}^{(i)}$ be the functions defined by \eqref{defi-t-0}-\eqref{defi-q-i}. 

\smallskip

\noindent
\textbf{(1)} For all $\beta_1,\beta_2 \in (0,2)$ such that $\beta_1+\beta_2>1$, it holds that
$$
\int_{\R^2} dx_1dx_2 \, \frac{|\mathcal{Q}^{(i)}(x_1,x_2)|^2}{|x_1|^{\beta_1-1}|x_2|^{\beta_2-1}} \ < \ \infty \, .
$$

\smallskip

\noindent
\textbf{(2)} For all $\la_1,\la_2\in (0,2)$ it holds that
$$
\int_{|x_1|\leq 1} dx_1\int_{\R}dx_2 \, \frac{|\mathcal{T}^{(i)}(x_1+x_2)|^2}{|x_1|^{\la_1-1}|x_2|^{\la_2-1}} \ < \ \infty\,  .
$$

\smallskip

\noindent
\textbf{(3)} For all $\la_1>0$ and $\la_2 \in (0,2)$ such that $\la_1+\la_2 >3$, it holds that
$$
\int_{|x_1|\geq 1} dx_1\int_{\R}dx_2 \, \frac{|\mathcal{T}^{(i)}(x_1+x_2)|^2}{|x_1|^{\la_1-1}|x_2|^{\la_2-1}} \ < \ \infty\,  .
$$
\end{lemma}

\begin{proof}
The result of item \textbf{(1)} is borrowed from \cite[Lemma 3.11]{De16}.

\smallskip

As for the proofs of items \textbf{(2)} and \textbf{(3)},  they both rely on the readily-checked bound
$$|\mathcal{T}^{(i)}(x)|^2\lesssim \frac{1}{1+|x|^2} .$$
For \textbf{(2)}, we have 
\begin{align*}
&\int_{|x_1|\leq 1} dx_1\int_{\R}dx_2 \, \frac{|\mathcal{T}^{(i)}(x_1+x_2)|^2}{|x_1|^{\la_1-1}|x_2|^{\la_2-1}}\lesssim \int_{|x_1|\leq 1} dx_1\int_{\R}dx_2 \, \frac{1}{|x_1|^{\la_1-1}|x_2|^{\la_2-1}}\frac{1}{1+|x_1+x_2|^2}\\
&\hspace{3cm}\lesssim \int_{|x_1|\leq 1} \frac{dx_1}{|x_1|^{\la_1-1}}\int_{|x_2|\leq 2} \, \frac{dx_2}{|x_2|^{\la_2-1}}+\int_{|x_1|\leq 1} \frac{dx_1}{|x_1|^{\la_1-1}}\int_{|x_2|\geq 2} \, \frac{dx_2}{|x_2|^{\la_2+1}} \ < \ \infty \, .
\end{align*}
As for \textbf{(3)}, we can first write
\begin{align*}
&\int_{|x_1|\geq 1} dx_1\int_{\R}dx_2 \, \frac{|\mathcal{T}^{(i)}(x_1+x_2)|^2}{|x_1|^{\la_1-1}|x_2|^{\la_2-1}}\\
&\lesssim \int_{|x_1|\geq 1} dx_1\int_{|x_2|\leq \frac12}dx_2 \, \frac{1}{|x_1|^{\la_1+1}|x_2|^{\la_2-1}}+\int_{|x_1|\geq 1} dx_1\int_{|x_2|\geq \frac12}dx_2 \, \frac{1}{|x_1|^{\la_1-1}|x_2|^{\la_2-1}}\frac{1}{1+|x_1+x_2|^2}.
\end{align*}
The first integral is clearly finite. Then decompose the second integral as
\begin{align}
&\int_{|x_1|\geq 1} dx_1\int_{|x_2|\geq \frac12}dx_2 \, \frac{1}{|x_1|^{\la_1-1}|x_2|^{\la_2-1}}\frac{1}{1+|x_1+x_2|^2}\nonumber\\
&=\int_{|x_1|\geq 1} dx_1\int_{\{\frac12 \leq |x_2|\leq \frac12|x_1| \}\cup \{|x_2|\geq \frac32 |x_1|\}}dx_2 \, \frac{1}{|x_1|^{\la_1-1}|x_2|^{\la_2-1}}\frac{1}{1+|x_1+x_2|^2}\nonumber\\
&\hspace{2cm}+\int_{|x_1|\geq 1} dx_1\int_{\frac12|x_1|\leq |x_2|\leq \frac32 |x_1|}dx_2 \, \frac{1}{|x_1|^{\la_1-1}|x_2|^{\la_2-1}}\frac{1}{1+|x_1+x_2|^2}\label{amelior-est-t-proof}
\end{align}
Now, on the one hand, note that if $\frac12 \leq |x_2|\leq \frac12|x_1|$ or $|x_2|\geq \frac32 |x_1|$, then $|x_1+x_2|\geq \max \frac13\big(|x_1|,|x_2|\big)$, and so, for any $\beta\in [0,1]$\begin{align}
&\int_{|x_1|\geq 1} dx_1\int_{\{\frac12 \leq |x_2|\leq \frac12|x_1| \}\cup \{|x_2|\geq \frac32 |x_1|\}}dx_2 \, \frac{1}{|x_1|^{\la_1-1}|x_2|^{\la_2-1}}\frac{1}{1+|x_1+x_2|^2}\nonumber\\
&\lesssim \int_{|x_1|\geq 1} \frac{dx_1}{|x_1|^{\la_1+2\beta-1}}\int_{|x_2|\geq \frac12}\frac{dx_2}{|x_2|^{\la_2+2(1-\beta)-1}}\label{integr-qcq}
\end{align}
Due to the assumption $\la_1+\la_2>3$, we can obviously write $\la_1+\la_2>2+\varepsilon$ for any small $\varepsilon>0$, and from here we can pick $\beta:=\frac{\la_2}{2}-\frac{\varepsilon}{2}\in [0,1]$, so that $\la_2+2(1-\beta)-1=1+\varepsilon>1$ and $\la_1+2\beta-1=\la_1+\la_2-\varepsilon-1>1$. For such a value of $\beta$, both integrals in \eqref{integr-qcq} are thus finite.

\smallskip

On the other hand, we can write
\begin{align*}
&\int_{|x_1|\geq 1} dx_1\int_{\frac12|x_1|\leq |x_2|\leq \frac32 |x_1|}dx_2 \, \frac{1}{|x_1|^{\la_1-1}|x_2|^{\la_2-1}}\frac{1}{1+|x_1+x_2|^2}\\
&=\int_{|x_1|\geq 1} dx_1 \, x_1\int_{\frac12\leq |r|\leq \frac32 }dr \, \frac{1}{|x_1|^{\la_1+\la_2-2}|r|^{\la_2-1}}\frac{1}{1+|x_1|^2(1+r)^2}\\
&\lesssim \int_{|x_1|\geq 1} \frac{dx_1}{|x_1|^{\la_1+\la_2-2-\varepsilon}}\int_{\frac12\leq |r|\leq \frac32 }\frac{dr}{(1+r)^{1-\varepsilon}}.
\end{align*}
Using the assumption $\la_1+\la_2>3$, we can pick $\varepsilon >0$ small enough such that $\la_1+\la_2-2-\varepsilon>1$, which shows that the above quantity is finite.
Going back to \eqref{amelior-est-t-proof}, this achieves the proof of item~\textbf{(3)}. 

\end{proof}

With those notations and preliminary results in hand, let us go back to \eqref{cs-ell}. Invoking~\eqref{bou-var-psi-1-1} and \eqref{bou-var-psi-1-2}, our claim \eqref{eq:bound for tilde sc^l} amounts to show that
\begin{multline}
\cj^{0,\ell}:=\iint_{\R^{d+1}\times \R^{d+1}} d\la d\xi d\lati d\xiti \, |\cf K(2^{2\ell}\la,2^\ell \xi)|^2
 \cn_{H_0,\bh}(\la,\xi) \cn_{H_0,\bh}(\lati,\xiti) \\
\times 
\left( |\la|^2\mathcal{Q}^{(0)}(\la,\lati)^2\right)
\,\prod_{i=1}^d\left(\mathcal{T}^{(i)}(\xi_i+\xiti_i)\right)^2  
\lesssim  
2^{-4\ell (1-\varepsilon)},\label{tech-bou-1-1}
\end{multline}
and that for every fixed $i=1,\ldots,d$, we have
\begin{multline}
\mathcal{J}^{i,\ell}:=\iint_{\R^{d+1}\times \R^{d+1}} d\la d\xi d\lati d\xiti \, |\cf K(2^{2\ell}\la,2^\ell \xi)|^2 \cn_{H_0,\bh}(\la,\xi) \cn_{H_0,\bh}(\lati,\xiti) \left(\mathcal{T}^{(0)}(\lati)\right)^2  
\\
\times  
\prod_{j=1}^{i-1}\left(\mathcal{T}^{(j)}(\xiti_j)\right)^2 
\left(|\xi_i|^2 \mathcal{Q}^{(i)}(\xi_i,\xiti_i)^2\right)
\prod_{j=i+1}^d\left(\mathcal{T}^{(j)}(\xi_{j}+\xiti_{j})\right)^2
 \lesssim  2^{-4\ell (1-\varepsilon)}\, .\label{tech-bou-1-2}
\end{multline}

\

To establish these bounds, we will split the integration domain for the variables $\la,\xi$ along 
$$D_{-}:=\{\la\in \R: \, |\la|\leq 1\} \quad \text{and} \quad D_+:=\{\la\in \R: \, |\la| \geq 1\}\, ,$$
that is we set, for every $\mathbf{s}\in \{-,+\}^{d+1}$, $D_{\mathbf{s}}:=\prod_{k=0}^d D_{\mathbf{s}_k}$, and then consider 
\begin{multline}
\cj^{0,\ell}_{\mathbf{s}}:=\iint_{D_{\mathbf{s}}\times \R^{d+1}} d\la d\xi d\lati d\xiti \, |\cf K(2^{2\ell}\la,2^\ell \xi)|^2
 \cn_{H_0,\bh}(\la,\xi) \cn_{H_0,\bh}(\lati,\xiti) \\
\times 
\left( |\la|^2\mathcal{Q}^{(0)}(\la,\lati)^2\right)
\,\prod_{i=1}^d\left(\mathcal{T}^{(i)}(\xi_i+\xiti_i)\right)^2  ,\label{tech-bou-1-1-s} .
\end{multline}
For every fixed $i=1,\ldots,d$, we also set
\begin{multline}
\mathcal{J}^{i,\ell}_{\mathbf{s}}:=\iint_{D_{\mathbf{s}}\times \R^{d+1}} d\la d\xi d\lati d\xiti \, |\cf K(2^{2\ell}\la,2^\ell \xi)|^2 \cn_{H_0,\bh}(\la,\xi) \cn_{H_0,\bh}(\lati,\xiti) \left(\mathcal{T}^{(0)}(\lati)\right)^2  
\\
\times  
\prod_{j=1}^{i-1}\left(\mathcal{T}^{(j)}(\xiti_j)\right)^2 
\left(|\xi_i|^2 \mathcal{Q}^{(i)}(\xi_i,\xiti_i)^2\right)
\prod_{j=i+1}^d\left(\mathcal{T}^{(j)}(\xi_{j}+\xiti_{j})\right)^2
\, .\label{tech-bou-1-2-s}
\end{multline}
It is clear that \eqref{tech-bou-1-1} and \eqref{tech-bou-1-2} will hold true if we can show that for every $\mathbf{s}\in \{-,+\}^{d+1}$,
\begin{equation}\label{purpose}
\mathcal{J}^{0,\ell}_{\mathbf{s}} \lesssim 2^{-4\ell (1-\varepsilon)}\quad \text{and} \quad \mathcal{J}^{i,\ell}_{\mathbf{s}}\lesssim 2^{-4\ell (1-\varepsilon)} \, .
\end{equation}
We will now treat the two integrals \eqref{tech-bou-1-1-s} and \eqref{tech-bou-1-2-s} separately.

\

\noindent
\underline{Bound on \eqref{tech-bou-1-1-s}}. Let $\mathbf{s}\in \{-,+\}^{d+1}$ be fixed. We can apply Lemma \ref{lem:estim-fourier-k} and recall the definition~\eqref{notation:cn-h} of $\cn_{H_0,\bh}$ in order to assert that for all $a_0,a_1,\ldots,a_d\in [0,1]$ such that $a_0+a_1+\ldots+a_d <1$, the integral in~\eqref{tech-bou-1-1-s} is bounded (up to a  constant) by
\begin{multline}
2^{-4\ell(a_0+a_1+\ldots+a_d)}\bigg(\int_{D_{\mathbf{s}_0}\times \R} d\la d\lati  \, \frac{\mathcal{Q}^{(0)}(\la,\lati)^2}{|\la|^{(2a_0+2H_0-2)-1} |\lati|^{2H_0-1}} \bigg) \\
\times
\prod_{i=1}^d \bigg( \int_{D_{\mathbf{s}_i}\times \R} d\xi_i d\xiti_i \, \frac{\mathcal{T}^{(i)}(\xi_{i}+\xiti_{i})^2}{|\xi_i|^{4a_i+2H_i-1} |\xiti_i|^{2H_i-1}} \bigg) \, .\label{bou-varia-xi-1}
\end{multline}
The whole point now is that we can find parameters $a_0,a_1,\ldots,a_d\in [0,1]$ such that $a_0+a_1+\ldots+a_d =1-\varepsilon$ and such that the integrals involved in the above expression are all finite. In order to justify this claim, we can refer to Lemma \ref{lem:t-i-q-i-improv}. According to this property, the first integral  in \eqref{bou-varia-xi-1} is finite whenever $2a_0+2H_0>3-2H_0$ and $2a_0+2H_0<4$. Moreover, since $0<a_0<1$, we have $2H_0<2a_0+2H_0<2+2H_0<4$. Summarizing those elementary considerations and similar ones for the second integral in \eqref{bou-varia-xi-1}, we get that \eqref{bou-varia-xi-1} is a finite expression as long as
\begin{equation}\label{cond-a-1-1}
\left\{
\begin{array}{ll}
\max(2H_0,3-2H_0)<2a_0+2H_0<2+2H_0 &\\
2H_i<4a_i+2H_i <2 & \text{for} \ i\in \{i\in \{1,\ldots,d\}: \, \mathbf{s}_i=-\}\\
3-2H_i<4a_i+2H_i <4+2H_i& \text{for} \ i\in \{i\in \{1,\ldots,d\}: \, \mathbf{s}_i=+\}.
\end{array}
\right.
\end{equation}

Provided \eqref{cond-a-1-1} is met and $a_0+a_1+\ldots+a_d =1-\varepsilon$, we thus have that the expression \eqref{bou-varia-xi-1} is bounded, up to a constant, by $2^{4\ell(1-\varepsilon)}$. This proves \eqref{tech-bou-1-1}.

We now show that the above-reported conditions can indeed be fulfilled under our standing assumptions. In fact,

\smallskip

\noindent
(i) Since $H_0>\frac14$ (see \eqref{impli-h-0}), the first condition in \eqref{cond-a-1-1} is easily shown to be satisfied for some values of $a_0\in(0,1)$. 

\smallskip

\noindent
(ii) The conditions \eqref{cond-a-1-1} can also be made consistent with the desired assumption $\sum_{i=0}^da_i=1-\epsilon$ for $\epsilon>0$. In order to verify this assertion, sum the constraints in \eqref{cond-a-1-1}. This yields
\begin{align}\label{condit-a}
A^0_{\mathbf{s}}< 2(2a_0+2H_0)+\sum_{i=1}^d (4a_i+2H_i)<B^0_{\mathbf{s}}\, ,
\end{align}
with two parameters $A_{\mathbf{s}}, B_{\mathbf{s}}$ defined by
\begin{align}\label{defi-a-i}
A^0_{\mathbf{s}}&:= 2\max(2H_0,3-2H_0)+2\sum_{\substack{i=1,\ldots,d\\\mathbf{s}_i=-}} H_i+\sum_{\substack{i=1,\ldots,d\\\mathbf{s}_i=+}}(3-2H_i) 
\end{align}
\begin{align*}
B^0_{\mathbf{s}}&=2(2+2H_0)+2\big|\{i\in \{1,\ldots,d\}: \, \mathbf{s}_i=-\}\big|+\sum_{\substack{i=1,\ldots,d\\\mathbf{s}_i=+}} (4+2H_i)\, .
\end{align*}
We now resort to the assumption $\sum_{i=0}^da_i=1-\epsilon$. Recalling our notation $H=\sum_{i=1}^dH_i$, we end up with the condition
\begin{equation}\label{sum cond-a}
A^0_{\mathbf{s}}<4(1-\varepsilon)+2(2H_0+H) <B^0_{\mathbf{s}}.
\end{equation}
In order to see that these two inequalities are indeed satisfied (at least for $\varepsilon>0$ small enough), observe first that
\begin{align*}
B^0_{\mathbf{s}}&=2(2+2H_0)+2d+2\sum_{\substack{i=1,\ldots,d\\\mathbf{s}_i=+}} (1+H_i)\geq 4+2(2H_0+d)>4+2(2H_0+H)\, ,
\end{align*}
where the last inequality immediately follows from the trivial bound $H<d$.

\smallskip

As for the first inequality in \eqref{sum cond-a}, note that
\begin{align}
A^0_{\mathbf{s}}&< 2\max(2H_0,3-2H_0)+2\sum_{\substack{i=1,\ldots,d\\\mathbf{s}_i=-}} H_i+\sum_{\substack{i=1,\ldots,d\\\mathbf{s}_i=+}}(3-2H_i)\nonumber\\
&<2\max(2H_0,3-2H_0)+\sum_{i=1}^d \max(2H_i,3-2H_i)\nonumber\\
& <2\max(2,3-2H_0)+\max(2,3-2H_1)+\max(2,3-2H_2)\1_{d\geq 2}+2(d-2)\1_{d\geq 2} \, ,\label{bound-a-s}
\end{align}
where we have used the observation \eqref{observ-h-i} to derive the last inequality. The following table collects the possible values of the bound in \eqref{bound-a-s}, depending on $H_0,H_1,H_2$ (remember that $H_1\leq H_2$):

\begin{center}
\begin{tabular}{c|c|c|c|c}
$H_0$ & $H_1$ & $H_2$ &$A^0_{\mathbf{s}}$ for $d=1$&$A^0_{\mathbf{s}}$ for $d\geq 2$\\
\hline
$(0,\frac12]$ & $(0,\frac12]$ & $(0,\frac12]$ & $<9-2(2H_0+H_1)$ &$<2d+8-2(2H_0+H_1+H_2)$\\
\hline
$(0,\frac12]$ & $(0,\frac12]$ & $(\frac12,1)$ & $<9-2(2H_0+H_1)$&$<2d+7-2(2H_0+H_1)$\\
\hline
$(0,\frac12]$ & $(\frac12,1)$ & $(\frac12,1)$ & $<8-4H_0$&$<2d+6-4H_0$\\
\hline
$(\frac12,1)$ & $(0,\frac12]$ & $(0,\frac12]$ & $<7-2H_1$&$<2d+6-2(H_1+H_2)$\\
\hline
$(\frac12,1)$ & $(0,\frac12]$ & $(\frac12,1)$ & $<7-2H_1$&$<2d+5-2H_1$\\
\hline
$(\frac12,1)$ & $(\frac12,1)$ & $(\frac12,1)$ & $<6$&$<2d+4$\\
\end{tabular}
\end{center}
Based on these values, and using the three conditions \eqref{impli-h-0}-\eqref{impli-h-1}-\eqref{impli-h-2}, we can easily conclude that
$$A^0_{\mathbf{s}}< 2d+5<4+2(2H_0+H) \, ,$$
where the last bound is derived from the assumption $2H_0+H>d+\frac12$.

\smallskip

We have thus checked that \eqref{sum cond-a} holds true, and this completes the proof of the desired estimate
\begin{equation}\label{estim-fin-j-0-ell-s}
\mathcal{J}^{0,\ell}_{\mathbf{s}} \lesssim 2^{-4\ell (1-\varepsilon)}\, .
\end{equation}

\

\noindent
\underline{Bound on \eqref{tech-bou-1-2-s}}. Let us fix $\mathbf{s}\in \{-,+\}^{d+1}$ and $i\in \{1,\ldots,d\}$. In order to bound $\cj^{i,\ell}_{\mathbf{s}}$, we proceed similarly to \eqref{bou-varia-xi-1}. Namely we apply Lemma \ref{lem:estim-fourier-k} to assert that for all $a_0,a_1,\ldots,a_d\in [0,1]$ such that $a_0+a_1+\ldots+a_d <1$,
\begin{align}
\cj^{i,\ell}_{\mathbf{s}}&\lesssim 2^{-4\ell(a_0+a_1+\ldots+a_d)} \bigg( \int_{D_{\mathbf{s}_0}} \,  \frac{d\la}{|\la|^{2a_0+2H_0-1}}\bigg)\bigg( \int_{\R} d\lati \, \frac{\mathcal{T}^{(0)}(\lati)^2}{|\lati|^{2H_0-1}} \bigg) \prod_{r=1}^{i-1} \bigg( \int_{\R} d\xiti_r \, \frac{\mathcal{T}^{(r)}(\xiti_r)^2}{|\xiti_r|^{2H_r-1}} \bigg)\nonumber\\
&\ \times\prod_{k=1}^{i-1}\bigg( \int_{D_{\mathbf{s}_k}} \frac{d\xi_k}{|\xi_k|^{4a_k+2H_k-1}} \bigg)\nonumber\\
&\ \ \times \bigg(\int_{D_{\mathbf{s}_i}\times \R} d\xi_i d\xiti_i \, \frac{\mathcal{Q}^{(i)}(\xi_i,\xiti_i)^2}{|\xi_i|^{(4a_i+2H_i-2)-1}|\xiti_i|^{2H_i-1}} \bigg)\prod_{p=i+1}^d \bigg( \int_{D_{\mathbf{s}_p}\times \R} d\xi_p d\xiti_p \, \frac{\mathcal{T}^{(p)}(\xi_{p}+\xiti_{p})^2}{|\xi_p|^{4a_p+2H_p-1} |\xiti_p|^{2H_p-1}} \bigg) \, ,\label{bound J^il_j+}
\end{align}
where we recall that $D_-:=[-1,1]$ and $D_+:=\R\backslash [-1,1]$.

\smallskip

Based on  the criteria of Lemma \ref{lem:t-i-q-i-improv}, we get the following conditions on the parameters $a_0,a_1,\ldots,a_d$ (so as to ensure that the integrals in \eqref{bound J^il_j+} are all finite, and also that each $a_i$ belongs to $(0,1)$):
\begin{equation}\label{cond-a-1-2}
\left\{
\begin{array}{ll}
2H_0<2a_0+2H_0<2& \text{if} \ \mathbf{s}_0=-\\
2<2a_0+2H_0<2+2H_0& \text{if} \ \mathbf{s}_0=+ \\
2H_k<4a_k+2H_k <2 & \text{for} \ k\in \{k\in \{1,\ldots,i-1\}: \, \mathbf{s}_k=-\}\\
2<4a_k+2H_k <4+2H_k& \text{for} \ k\in \{k\in \{1,\ldots,i-1\}: \, \mathbf{s}_k=+\} \\
\max(2H_i,3-2H_i)<4a_i+2H_i <4\\
2H_p<4a_p+2H_p <2 & \text{for} \ p\in \{p\in \{i+1,\ldots,d\}: \, \mathbf{s}_p=-\}\\
3-2H_p<4a_p+2H_p <4+2H_p& \text{for} \ p\in \{p\in \{i+1,\ldots,d\}: \, \mathbf{s}_p=+\} .
\end{array}
\right.
\end{equation}

\smallskip

As in the proof of \eqref{tech-bou-1-1}, we still have to verify that the parameters $a_0,\dots,a_d$ can be chosen so that $\sum_{k=0}^da_k=1-\varepsilon$. To this aim, we use the same strategy as for \eqref{cond-a-1-1}. Namely we sum all the constraints in \eqref{cond-a-1-2}, which yields the following condition:
\begin{align}\label{condit-a-i}
A^i_{\mathbf{s}}< 4(1-\varepsilon)+2(2H_0+H)<B^i_{\mathbf{s}}\, ,
\end{align}
with two parameters $A^i_{\mathbf{s}}, B^i_{\mathbf{s}}$ defined by
\begin{align}
A^i_{\mathbf{s}}&:= 4\{H_0\, \1_{\mathbf{s}_0=-}+\1_{\mathbf{s}_0=+}\}+2\sum_{\substack{k=1,\ldots,i-1\\\mathbf{s}_k=-}} H_k+2\big|\{k\in \{1,\ldots,i-1\}: \, \mathbf{s}_k=+\}\big|\nonumber\\
&\hspace{1cm}+\max(2H_i,3-2H_i)+2\sum_{\substack{p=i+1,\ldots,d\\\mathbf{s}_p=-}} H_p+\sum_{\substack{p=i+1,\ldots,d\\\mathbf{s}_p=+}} (3-2H_p),\label{defi-a-+-i}
\end{align}
\begin{align*}
B^i_{\mathbf{s}}&:=4\{\, \1_{\mathbf{s}_0=-}+(1+H_0)\1_{\mathbf{s}_0=+}\}+2\big|\{k\in \{1,\ldots,i-1\}: \, \mathbf{s}_k=-\}\big|+\sum_{\substack{k=1,\ldots,i-1\\\mathbf{s}_k=+}} (4+2H_k)\\
&\hspace{3cm}+4+2\big|\{p\in \{i+1,\ldots,d\}: \, \mathbf{s}_p=-\}\big|+\sum_{\substack{p=i+1,\ldots,d\\\mathbf{s}_p=+}} (4+2H_p)\, .
\end{align*}
In order to see that $A^i_{\mathbf{s}}<4+2(2H_0+H)$, observe first that 
\begin{equation}\label{bound-a-i-s}
A^i_{\mathbf{s}}<4+2(i-1)+\sum_{q=i}^d\max(2H_q,3-2H_q) \, .
\end{equation}
Let us recall that, by \eqref{observ-h-i}, one has $H_q>\frac12$ for $q\geq 3$, and so the above bound yields, for $i\geq 3$, 
$$A^i_{\mathbf{s}}<4+2(i-1)+2(d-i+1)=4+2d<3+2(2H_0+H),$$
where we have used the assumption $d+\frac12<2H_0+H$ to derive the last inequality.

\smallskip

Then, using again \eqref{bound-a-i-s}, we have
\begin{align*}
A^2_{\mathbf{s}}<6+\max(2H_2,3-2H_2)+2(d-2)&=2+2d+\max(2H_2,3-2H_2)\\
&<5+2d<4+2(2H_0+H) \, ,
\end{align*}
where we have again used the assumption $d+\frac12<2H_0+H$ to derive the last inequality.

\smallskip

As for $A^1_{\mathbf{s}}$, we get by \eqref{bound-a-i-s} that
\begin{align*}
A^1_{\mathbf{s}}&<4+\max(2H_1,3-2H_1)+\max(2H_2,3-2H_2)+2(d-2)\\
&<2d+\max(2,3-2H_1)+\max(2,3-2H_2)\\
&\leq 2d+\max(4,5-2H_1,5-2H_2,6-2(H_1+H_2))\\
&<5+2d<4+2(2H_0+H) \, ,
\end{align*}
where we have used \eqref{impli-h-2} to get the fourth inequality.

\smallskip

This completes the proof of the first inequality in \eqref{condit-a-i}.

\smallskip

For the second inequality (i.e., $4(1-\varepsilon)+2(2H_0+H)<B^i_{\mathbf{s}}$), let us write $B^i_{\mathbf{s}}$ as
\begin{align*}
B^i_{\mathbf{s}}&=4\{\, \1_{\mathbf{s}_0=-}+(1+H_0)\1_{\mathbf{s}_0=+}\}+2(i-1)+2\sum_{\substack{k=1,\ldots,i-1\\\mathbf{s}_k=+}} (1+H_k)\\
&\hspace{3cm}+4+2(d-i)+2\sum_{\substack{p=i+1,\ldots,d\\\mathbf{s}_p=+}} (1+H_p)\\
&=2d+2+4\{\, \1_{\mathbf{s}_0=-}+(1+H_0)\1_{\mathbf{s}_0=+}\}+2\sum_{\substack{k=1,\ldots,i-1\\\mathbf{s}_k=+}} (1+H_k)+2\sum_{\substack{p=i+1,\ldots,d\\\mathbf{s}_p=+}} (1+H_p)\, ,
\end{align*}
and from here it is clear that
$$B^i_{\mathbf{s}}>6+2d\geq 4+2(2H_0+H),$$
where the last inequality stems from the assumption $2H_0+H\leq d+1$.

\smallskip

We have thus checked that \eqref{condit-a-i} holds true, and this completes the proof of the desired estimate: for every $i=1,\ldots,d$,
\begin{equation}\label{estim-fin-j-i-ell-s}
\mathcal{J}^{i,\ell}_{\mathbf{s}} \lesssim 2^{-4\ell (1-\varepsilon)}\, .
\end{equation}

\smallskip

The combination of \eqref{estim-fin-j-0-ell-s} and \eqref{estim-fin-j-i-ell-s} precisely corresponds to \eqref{purpose}, and accordingly the proof of \eqref{eq:bound for tilde sc^l} is achieved.

\subsection{Conclusion: proof of Theorem \ref{theo:conv-k-rp}.}\label{subsec:conclusion}

Let us now see how we can use the moments estimates of Propositions \ref{prop:estim-mom-first} and \ref{prop:estim-mom-second} in order to prove the desired convergence \eqref{desired-convergence}. 

\smallskip

First, by applying Lemma \ref{GRR-gene} to a constant distribution $\zeta_{s,x}:=\dw^n-\dw^m$ (which means that $\theta^{i}=\zeta^{\sharp,i}=0$ in Lemma~\ref{GRR-gene}), we get that for every $k,p\geq 1$ ,
\begin{align*}
\mathbb{E} \Big[ \big\|\dw^{n}-\dw^{m}\big\|_{\al;k,w}^{2p} \Big]& \lesssim \mathbb{E} \Big[\sup_{\psi\in \Psi} \sup_{\ell\geq 0} \sup_{(s,x)\in \Lambda_\scal^\ell \cap([-(k+2),k+2]\times \R^d) } 2^{2\ell p \al} \frac{|\langle \dw^{n}-\dw^{m},\cs_{s,x}^{2^{-\ell}} \psi \rangle |^{2p}}{w(x)^{2p}}\Big]\\
&\lesssim \sum_{\psi\in \Psi} \sum_{\ell \geq 0} \sum_{(s,x)\in \Lambda^\ell_\scal \cap([-(k+2),k+2]\times \R^d)}2^{2\ell p \al} \frac{\mathbb{E}\big[|\langle \dw^{n}-\dw^{m},\cs_{s,x}^{2^{-\ell}} \psi \rangle |^{2p}\big]}{w(x)^{2p}} \ , 
\end{align*}
Furthermore, $\dw^{n}-\dw^{m}$ is a Gaussian process. Therefore we have
\begin{align}
&\mathbb{E} \Big[ \big\|\dw^{n}-\dw^{m}\big\|_{\al;k,w}^{2p} \Big] 
\lesssim \sum_{\psi\in \Psi} \sum_{\ell \geq 0} \sum_{(s,x)\in \Lambda^\ell_\scal \cap([-(k+2),k+2]\times \R^d)}2^{2\ell p \al} \frac{\mathbb{E}\big[|\langle \dw^{n}-\dw^{m},\cs_{s,x}^{2^{-\ell}} \psi \rangle |^{2}\big]^p}{w(x)^{2p}}\nonumber\\
& \hspace{2cm}
\lesssim  2^{-m\varepsilon p}\sum_{\ell \geq 0}2^{2\ell p (\al+d+2-(2H_0+H)+\varepsilon)} \sum_{(s,x)\in \Lambda^\ell_\scal \cap([-(k+2),k+2]\times \R^d)} w(x)^{-2p} \, ,\label{estim-interm}
\end{align}
where the last inequality follows from Proposition \ref{prop:estim-mom-first} and the fact that $\Psi$ is a finite set.

\smallskip

At this point, observe that
\begin{align*}
\sum_{(s,x)\in \Lambda^\ell_\scal \cap([-(k+2),k+2]\times \R^d)} w(x)^{-2p}&=\bigg(\sum_{q_0 \in \Z}\1_{\{-(k+2) \leq q_0 2^{-2\ell}\leq k+2\}}\bigg) \bigg( \sum_{q\in \Z^d} \big(1+2^{-\ell}|q|\big)^{-2\ka p}\bigg) \\
&\lesssim 2^{2\ell} k \bigg\{1+2^{\ka \ell p} \sum_{q\in \Z^d \backslash \{0\}} |q|^{-2\ka p} \bigg\} \ .
\end{align*}
Owing to our assumption $\al<-(d+2)+(2H_0+H)$, we can pick $\varepsilon >0$ small enough such that $\beta:= -\al-(d+2)+(2H_0+H)-\varepsilon >0$. Going back to \eqref{estim-interm}, we have obtained that for every $k,p\geq 1$,
\begin{equation}\label{estim-first-last}
\mathbb{E} \Big[ \big\|\dw^{n}-\dw^{m}\big\|_{\al;k,w}^{2p} \Big]  \lesssim k \, 2^{-m\varepsilon p} \sum_{\ell \geq 0}   \bigg\{2^{-2\ell (\beta p-1)}+2^{-2\ell ((\beta-\ka) p-1)} \sum_{q\in \Z^d \backslash \{0\}} |q|^{-2\ka p} \bigg\} \ .
\end{equation}
Without loss of generality, we can here assume that $0< \ka <\beta$. Then we can pick $p\geq 1$ large enough so that $(\beta-\ka)p-1 >0$ and $2\ka p> d$, which ensures that the sum in \eqref{estim-first-last} is finite, and so, for every $k\geq 1$ and any such large $p\geq 1$,
\begin{equation}\label{ordre-un}
\mathbb{E} \Big[ \big\|\dw^{n}-\dw^{m}\big\|_{\al;k,w}^{2p} \Big]  \lesssim k \, 2^{-m\varepsilon p} \,  .
\end{equation}

\smallskip

Using similar arguments (starting from Lemma \ref{GRR-gene}, and also leaning on \eqref{bou-k-ast}), we can then turn the estimate of Proposition \ref{prop:estim-mom-second} into the bound 
\begin{equation}\label{ordre-deux}
\mathbb{E} \Big[ \big\|\widehat{\bw}^{\mathbf{2},n}-\widehat{\bw}^{\mathbf{2},m}\big\|_{2\al+2;k,w}^{2p} \Big]  \lesssim k \, 2^{-m\varepsilon p} \,  ,
\end{equation}
for every $k\geq 1$, every $\varepsilon >0$ small enough and every $p\geq 1$ large enough.

\smallskip

Combining \eqref{ordre-un} and \eqref{ordre-deux}, we get that for all $\varepsilon >0$ small enough and $p\geq 1$ large enough
$$\mathbb{E} \big[d_{\al,w}(\widehat{\uw}^{n},\widehat{\uw}^m)^{2p} \big]\lesssim 2^{-m\varepsilon p} \, ,$$
for all $n\geq m\geq 1$, and accordingly $(\widehat{\uw}^{n})_{n\geq 1}$ is a Cauchy sequence in $L^p(\Omega; (\mathcal{E}^{K}_{\al;w},d_{\al;w}))$. By Lemma~\ref{complet}, we can assert that there exists an element $\widehat{\uw}\in \mathcal{E}^{K}_{\al;w}$ satisfying 
$$\mathbb{E} \big[d_{\al,w}(\widehat{\uw},\widehat{\uw}^m)^{2p} \big]\lesssim 2^{-m\varepsilon p} \, ,$$
for every $p\geq 1$ large enough. The desired conclusion, that is the almost sure convergence of $\widehat{\uw}^{n}$ to $\widehat{\uw}$ in $(\mathcal{E}^{K}_{\al;w},d_{\al;w})$, immediately follows from Borel-Cantelli lemma.

\section{Proof of Theorem \ref{theo:conv-k-rp-spa}}\label{sec:stochastic-constructions-spa}

As we announced it earlier, the proof of Theorem \ref{theo:conv-k-rp-spa} will in fact reduce to a review of the few adaptations to be made with respect to the proof of Theorem \ref{theo:conv-k-rp}. Observe first that in this setting, identities \eqref{cova-xi-n-xi-n} and \eqref{cova-i-n-i-n} immediately give way to the following covariance formulas:

\begin{lemma}
Let $\dw^n$ be the smoothed noise defined by \eqref{def:approximate noise-spa} and recall that the kernel $\kti$ is defined by \eqref{decompo-k}. For every fixed $n\geq 1$, the families $\{\dw^n(y); y\in\R^{d}\}$ and $\{\kti\ast\dw^n(y); y\in\R^{d}\}$ are centered Gaussian processes with respective covariance functions given by the formulas
\begin{equation}\label{cova-xi-n-xi-n-spa}
\mathbb{E}\big[ \dw^n(y) \dw^n(\yti) \big]=c_{\bh}^2 \int_{\R^{d}} d\xi\, |\cfs\rho_n(\xi)|^2 \cn_{\bh}(\xi)e^{\imath \xi \cdot (y-\yti)}\, ,
\end{equation}
and
\begin{align}
&\mathbb{E}\big[ (\kti\ast\dw^n)(y) (\kti\ast\dw^n)(\yti) \big]=c_{\bh}^2 \int_{\R^{d}} d\xi\, |\cfs\rho_n(\xi)|^2 |\cfs \kti(\xi)|^2\cn_{\bh}(\xi)\lastchange{e^{\imath\xi \cdot (y-\yti)}}\, ,\label{cova-i-n-i-n-spa}
\end{align}
where the notation $\cn_{\bh}$ has been introduced in \eqref{notation:cn-h-spa} and the constant $c_{\bh}$ is the one given in \eqref{cstt-c-h}.
\end{lemma}

\subsection{Moment estimate for the first component} 

Morally, we need to check that the result of Proposition \ref{prop:estim-mom-first} still holds for $H_0=1$. In a more rigorous way, one has here:

\begin{proposition}\label{prop:estim-mom-first-spa}
For all $\ell\geq 0$, $n\geq m\geq 0$, $\psi\in \cb^{2(d+1)}_\scal$ and $x\in \R^{d}$, it holds that
\begin{align}\label{eq:dw^n-dw^m-spa}
\mathbb{E}\big[ |\langle \dw^{n}-\dw^m,\psiti^\ell_{x} \rangle |^{2}\big] \lesssim 2^{2\ell (d-H+\varepsilon)} 2^{-m\varepsilon}\, ,
\end{align}
where $\psiti(x):=\int_{\R}ds \, \psi(s,x)$, $\psiti^\ell_{x}(y):=2^{\ell d}\psiti(2^\ell (y-x))$, and the proportional constant in $\lesssim$ does not depend on $n,m,\ell,s,x$. 
\end{proposition}

\begin{proof}
It suffices to follow the arguments of the proof of Proposition \ref{prop:estim-mom-first}, and therein replace identity~\eqref{cova-xi-n-xi-n} with the covariance formula \eqref{cova-xi-n-xi-n-spa}.  
\end{proof}

\subsection{Moment estimate for the second component}

The preliminary estimates on $\cf K$ and $\cf R$ (i.e., Lemmas \ref{lem:estim-fourier-k} and \ref{lem:estim-fouri-r}) become estimates on $\cfs \kti$ and $\int_0^\infty ds \, \cfs R(s,.)$ in the spatial setting. Just as their space-time counterparts, these bounds follow from the analysis of the expansions contained in \eqref{decompo-k}.

\begin{lemma}\label{lem:estim-fourier-kti} 
Let $K$ be the localized heat kernel of Definition \ref{defi:loc-heat-ker}, and define $\kti$ along \eqref{kernel-k-tilde}. 
For all fixed $a_1,\ldots,a_d \in [0,1]$ such that $\sum_{i=1}^d a_i<1$, one has, for every $\xi\in \R^{d}$,
$$|\cfs \kti(\xi)|\lesssim \prod_{i=1}^d |\xi_{i}|^{-2a_i} \ .$$
\end{lemma}

\begin{lemma}\label{lem:estim-fouri-r-spa}
Let $R$ be the remainder term associated with the localized heat kernel $K$ (along Definition \ref{defi:loc-heat-ker}). Then, for all fixed $a_1,\ldots,a_d \geq 0$ such that $\sum_{i=1}^d a_i> 1$, one has, for every $\xi\in \R^{d+1}$,
\begin{equation}\label{estim-fouri-r-spa}
\bigg|\int_0^\infty ds \, \cfs R(s,.)(\xi)\bigg|\lesssim \prod_{i=1}^d |\xi|^{-2a_i} \ .
\end{equation}
As a consequence, if $\bh=(H_1,\ldots,H_d)\in (0,1)^{d}$ is such that $H < d-1$, it holds that
\begin{equation}\label{integr-r-spa}
\int_{\R^{d}} d\xi\, \cn_{\bh}(\xi) \bigg|\int_0^\infty ds \, \cfs R(s,.)(\xi)\bigg| < \infty .
\end{equation}
\end{lemma}

\smallskip

A similar decomposition to \eqref{decomp EW^2n_sx} can also be exhibited in this time-independent situation. 

\begin{lemma}\label{lem:decompo-esp-xi-2-spa} 
Let $\bw^{\mathbf{2},n}$ be the increment given by~\eqref{notation-xi-2-spa}, and recall that the renormalization constant $\mathfrak{c}^{(n)}_{\rho,\bh}$ is defined by \eqref{renormal-spa}. Then for all $x,y\in \R^{d}$ and $n\geq 1$, one has the decomposition
\begin{align}\label{decomp EW^2n_sx-spa}
\mathbb{E}\big[ \bw^{\mathbf{2},n}_{x}(y)\big]=\mathfrak{c}^{(n)}_{\rho,\bh}+\ce^n_{x}(y) \, ,
\end{align}
for some function $\ce^n_{x}$ such that for all $\varepsilon\in (0,1)$, $\ell\ge 0$ and  $\psi\in\cb_\scal^\ell$, we have
\begin{equation}\label{estim-mathcal-e-spa}
\big|\langle \ce^n_{x},\psiti^\ell_{x}\rangle \big| \lesssim  2^{2\ell (d-H-1+\varepsilon)} \, .
\end{equation}
Moreover, in relation \eqref{estim-mathcal-e-spa} the proportional constant does not depend on $n,\ell,x$.
\end{lemma}

\begin{proof}
We mimic the proof of Lemma \ref{lem:decompo-esp-xi-2}. First, one can of course write
$$
\mathbb{E}\big[ \bw^{\mathbf{2},n}_{x}(y)\big]=\mathfrak{c}^{(n)}_{\rho,\bh}+\ce^n_{x}(y),
$$ 
with
\begin{align}\label{eq:Epsi^n_sx-spa}
\ce^n_{x}(y):=\Big\{\tilde{Q}^n(y;y)-\mathfrak{c}^{(n)}_{\rho,\bh} \Big\}-\tilde{Q}^n(x;y) \quad \text{and} \quad \tilde{Q}^n(x;y):=\mathbb{E}\big[(\kti\ast \dw^n)(x) \dw^n(y) \big] \, .
\end{align}

On the one hand, using \eqref{cova-xi-n-xi-n-spa}-\eqref{cova-i-n-i-n-spa}, and along the same lines as for \eqref{first-bo}, we get
\begin{align*}
&\bigg| \int_{\R^{d}} dy\, \tilde{Q}^n(x;y)\psiti^\ell_{x}(y) \bigg| =c_{\bh}^2\, 2^{2\ell (d-H)} \bigg|\int_{\R^{d}} d\xi\, |\cfs\rho_n(2^\ell\xi)|^2 \cn_{\bh}(\xi)\cfs \kti(2^\ell\xi)\cfs\psiti(\xi) \bigg| \, .
\end{align*}
Since $H \leq d-1$, we can pick $a_1,\ldots,a_d$ in $[0,1]$ such that $\sum_{i=1}^d a_i=1-\varepsilon$ and $2H_i+2a_i-1<1$ for $i=1,\ldots,d$. Applying Lemma \ref{lem:estim-fourier-kti} with these parameters and invoking the inequality $|\cfs\rho_n(\xi)| \lesssim 1$, we deduce
\begin{align*}
&\bigg| \int_{\R^{d}} dy\, \tilde{Q}^n(x;y)\psiti^\ell_{x}(y) \bigg| \lesssim 2^{2\ell (d-H-1+\varepsilon)} \int_{\R^{d}} d\xi\, \prod_{i=1}^d \frac{1}{|\xi_i|^{2H_i+2a_i-1}} \big| \cfs\psiti(\xi) \big|\lesssim2^{2\ell (d-H-1+\varepsilon)}  \, .
\end{align*}

\smallskip

Then, to bound the difference $\tilde{Q}^n(y;y)-\mathfrak{c}^{(n)}_{\rho,\bh}$ in \eqref{eq:Epsi^n_sx-spa}, consider the two possible situations for $H$.

\smallskip

\noindent
\underline{First case: $H<d-1$.} In this case, going back to the definition \eqref{eq:cj-c-n-spa} of $\cj_{\rho,\bh}$, we can write
\begin{equation}\label{recall-cstt-renorm-spa}
\mathfrak{c}^{(n)}_{\rho,\bh}=2^{2n(d-H-1)}c_{\bh}^2 \cj_{\rho,\bh} =c_{H_0}^2 \int_{\R^{d}}  |\mathcal{F}^{\textsc s}{\rho_n}(\xi)|^2   \cn_{\bh}(\xi)\bigg( \int_0^\infty ds\, \mathcal{F}^{\textsc s} p_s(\xi)\bigg)\, d\xi \, .
\end{equation}
Besides, using \eqref{decompo-k}, it holds that
\begin{equation}\label{decompo-fouri-kti}
\cfs \kti(\xi)=\int_0^\infty ds \, \cfs K(s,.)(\xi)=\int_0^\infty ds \, \cfs p_s(\xi)-\int_0^\infty ds \, \cfs R(s,.)(\xi),
\end{equation}
and so, in light of \eqref{recall-cstt-renorm-spa},
\begin{align*}
&\tilde{Q}^n(y;y)-\mathfrak{c}^{(n)}_{\rho,\bh}
=-c_{\bh}^2 \int_{\R^{d}}  |\cfs\rho_n(\xi)|^2 \cn_{\bh}(\xi) \bigg(\int_0^\infty ds \, \cfs R(s,.)(\xi)\bigg) d\xi\, .
\end{align*}
Thus, thanks to \eqref{integr-r-spa} and to the uniform estimate $|\cfs\rho_n(\xi)| \lesssim 1$, we obtain
$$ 
\big|\tilde{Q}^n(y;y)-\mathfrak{c}^{(n)}_{\rho,\bh}\big| \lesssim 1 \le 2^{2\ell (d-H-1+\varepsilon)} \, .
$$

\smallskip

\noindent
\underline{Second case: $H=d-1$.} Due to the latter relation, it can be checked that
\begin{equation*}
\mathfrak{c}^{(n)}_{\rho,\bh}=c_{\bh}^2 \int_{|\xi|\geq 1}  |\mathcal{F}^{\textsc s}{\rho_n}(\xi)|^2   \cn_{\bh}(\xi)\bigg( \int_0^\infty ds\, \mathcal{F}^{\textsc s} p_s(\xi)\bigg)\, d\xi \, ,
\end{equation*}
and accordingly, by \eqref{decompo-fouri-kti},
\begin{align*}
&\tilde{Q}^n(y;y)-\mathfrak{c}^{(n)}_{\rho,\bh}\\
&=c_{\bh}^2 \bigg[\int_{|\xi|\leq 1}  |\cfs\rho_n(\xi)|^2 \cn_{\bh}(\xi) \cfs \kti(\xi) \, d\xi- \int_{|\xi|\geq 1}  |\cfs\rho_n(\xi)|^2 \cn_{\bh}(\xi) \bigg(\int_0^\infty ds \, \cfs R(s,.)(\xi)\bigg) \,  d\xi\bigg] \,.
\end{align*}
Using the results of Lemma \ref{lem:estim-fourier-kti} and Lemma \ref{lem:estim-fouri-r-spa}, we easily conclude that
\begin{align*}
& \big|\tilde{Q}^n(y;y)-\mathfrak{c}^{(n)}_{\rho,\bh}\big|\lesssim \int_{|\xi|\leq 1}  \cn_{\bh}(\xi) \,  d\xi+ \int_{|\xi|\geq 1}   \cn_{\bh}(\xi) \bigg|\int_0^\infty ds \, \cfs R(s,.)(\xi)\bigg| \, d\xi  \ \lesssim \ 1
\ \le \ 2^{2\ell \varepsilon} \, ,
\end{align*}
which corresponds to the desired bound in this case.
\end{proof}

The spatial counterpart of the central Proposition \ref{prop:estim-mom-second} now takes the following (expected) shape.
\begin{proposition}\label{prop:estim-mom-second-spa} 
Let $\widehat{\bw}^n$ be the renormalized $K$-rough path defined in the statement of Theorem \ref{theo:conv-k-rp-spa}. Then for all $\ell\geq 0$, $n\ge m \ge 0$, $\psi\in \cb^{2(d+1)}_\scal$, $x\in \R^{d}$ and $\varepsilon \in (0,1)$, it holds that
\begin{align}\label{eq:hatW^n-hatW^m-spa}
\mathbb{E}\Big[ |\langle \widehat{\bw}^{\mathbf{2},n}_{x}-\widehat{\bw}^{\mathbf{2},m}_{x},\psiti^\ell_{x} \rangle |^{2}\Big] \lesssim 2^{4\ell (d-H-1+\varepsilon)}2^{-m\varepsilon} \, ,
\end{align}
where the proportional constant in \eqref{eq:hatW^n-hatW^m-spa} does not depend on $n,m,\ell,x$.
\end{proposition}

\begin{proof}
Just as in the proof of Proposition \ref{prop:estim-mom-second}, we only focus on the proof of \eqref{eq:hatW^n-hatW^m-spa} for $m=0$.

Using the decomposition exhibited in Lemma \ref{lem:decompo-esp-xi-2-spa}, we get first
\begin{align*}
\mathbb{E}\big[ |\langle \widehat{\bw}^{\mathbf{2},n}_{x},\psiti^\ell_{x} \rangle |^{2}\big]&=\mathbb{E}\big[ |\langle \bw^{\mathbf{2},n}_{x}-\mathfrak{c}^{(n)}_{\rho,\bh},\psiti^\ell_{x} \rangle |^{2}\big]
=\big(\big\langle \mathcal{E}^n_{x},\psi^\ell_{x} \big\rangle  \big)^2+\mathcal{U}^{\ell,n}_{x}+\mathcal{V}^{\ell,n}_{x} \, ,
\end{align*}
where 
$$\mathcal{U}^{\ell,n}_{x}:=\iint_{\R^{d}\times \R^{d}} dy d\yti \, \psiti^\ell_{x}(y) \psiti^\ell_{x}(\yti)\mathbb{E}\big[\ci^n_{x}(y) \ci^n_{x}(\yti)  \big]\mathbb{E}\big[ \dw^{n}(y)\dw^{n}(\yti) \big]$$
and
$$\mathcal{V}^{\ell,n}_{x}:=\iint_{\R^{d}\times \R^{d}} dy d\yti \, \psiti^\ell_{x}(y) \psiti^\ell_{x}(\yti)\mathbb{E}\big[\ci^n_{x}(y)  \dw^{n}(\yti) \big]\mathbb{E}\big[\dw^{n}(y)\ci^n_{x}(\yti) \big] \ .$$
From here, and due to \eqref{estim-mathcal-e-spa}, the proof of \eqref{eq:hatW^n-hatW^m-spa} consists in checking that $|\mathcal{U}^{\ell,n}_{x} | +|\mathcal{V}^{\ell,n}_{x}| \lesssim 2^{4\ell (d-H-1+\varepsilon)}$.
In fact, we can follow line by line the arguments leading to \eqref{eq:uniform bound of UV by calS} (replacing of course \eqref{cova-xi-n-xi-n}-\eqref{cova-i-n-i-n} with \eqref{cova-xi-n-xi-n-spa}-\eqref{cova-i-n-i-n-spa}) to obtain that $|\mathcal{U}^{\ell,n}_{x}|+|\mathcal{V}^{\ell,n}_{x}|\lesssim 2^{4\ell(d-H)}\tilde{\cs}^\ell$, where
\begin{align}
\tilde{\cs}^\ell=
\iint_{\R^{d}\times \R^{d}} d\xi  d\xiti \, |\cf \kti(2^\ell \xi)|^2 &\cn_{\bh}(\xi) \cn_{\bh}(\xiti) \big| \cf\psiti(\xi+\xiti)-\cf\psiti(\xiti)\big|^2 \, .\label{cs-ell-spa}
\end{align}
Therefore, in view of \eqref{eq:hatW^n-hatW^m-spa}, it remains us to check that for any $\epsilon\in(0,1)$ we have
\begin{align}\label{eq:bound for tilde sc^l-spa}
\tilde{\cs}^\ell\lesssim 2^{-4\ell(1-\varepsilon)}. 
\end{align}
To this end, we can bound the difference $|\cf\psiti(\xi+\xiti)-\cf\psiti(\xiti)|$ in $\tilde{\cs}^\ell$ using the inequalities
\begin{align*}
&\big| \cf\psi(\xiti_1,\ldots,\xiti_{i-1},\xi_i+\xiti_i,\xi_{i+1}+\xiti_{i+1},\ldots,\xi_d+\xiti_d)\\
&\hspace{5cm}
-\cf\psi(\xiti_1,\ldots,\xiti_{i-1},\xiti_i,\xi_{i+1}+\xiti_{i+1},\ldots,\xi_d+\xiti_d)\big|\\
&\hspace{3cm}\lesssim \prod_{j=1}^{i-1} \mathcal{T}^{(j)}(\xiti_j) \,\left(|\xi_i| \cdot\mathcal{Q}^{(i)}(\xi_i,\xiti_i)\right) \,\prod_{j=i+1}^{d} \mathcal{T}^{(j)}(\xi_j+\xiti_{j})
 \, , \quad \quad i=1,\ldots,d,
\end{align*}
where, for $\lambda,\tilde{\lambda}\in\R$, the quantities $\mathcal{T}^{(i)}(\la)$ and $\mathcal{Q}^{(i)}(\la,\tilde{\la})$ are here defined by
\begin{equation}\label{defi-t-i-spa}
\mathcal{T}^{(i)}(\la):=\bigg(\int_{\R^{d}} dy \, |(\partial_{x_1}\cdots \partial_{x_d} \psiti)(y) | \bigg| \int_0^{y_i} dz_i \, e^{-\imath \la z_i}  \bigg|^{d}\bigg)^{1/d} \, ,
\end{equation}
\begin{equation}\label{defi-q-i-spa}
\mathcal{Q}^{(i)}(\la,\tilde{\la}):=\bigg(\int_{\R^{d}} dy \, |(\partial_{x_1}\cdots \partial_{x_d} \psi)(y) | \bigg| \int_0^{y_i} dz_i\int_0^{z_i} dw_i \, e^{-\imath \tilde{\la} z_i} e^{-\imath \la w_i} \bigg|^{d}\bigg)^{1/d} \, .
\end{equation}
With those notations, the claim \eqref{eq:bound for tilde sc^l-spa} reduces to showing that for every fixed $i=1,\ldots,d$, we have
\begin{multline}
\mathcal{J}^{i,\ell}:=\iint_{\R^{d}\times \R^{d}} d\xi  d\xiti \, |\cf \kti(2^\ell \xi)|^2 \cn_{\bh}(\xi) \cn_{\bh}(\xiti) 
\\
\times  
\prod_{j=1}^{i-1}\left(\mathcal{T}^{(j)}(\xiti_j)\right)^2 
\left(|\xi_i|^2 \mathcal{Q}^{(i)}(\xi_i,\xiti_i)^2\right)
\prod_{j=i+1}^d\left(\mathcal{T}^{(j)}(\xi_{j}+\xiti_{j})\right)^2
 \lesssim  2^{-4\ell (1-\varepsilon)}\, .\label{tech-bou-1-2-spa}
\end{multline}
Let us again follow the pattern of the proof of Proposition \ref{prop:estim-mom-second} and split the integration domain for the variables $\xi_1,\ldots,\xi_d$ along $D_{-}:=\{\la\in \R: \, |\la|\leq 1\}$ and $D_+:=\{\la\in \R: \, |\la| \geq 1\}$. In other words, we set, for every $\mathbf{s}\in \{-,+\}^{d}$, $D_{\mathbf{s}}:=\prod_{k=1}^d D_{\mathbf{s}_k}$,
and then consider, for every $i=1,\ldots,d$, 
\begin{multline*}
\mathcal{J}^{i,\ell}_{\mathbf{s}}:=\iint_{D_{\mathbf{s}}\times \R^{d}} d\xi  d\xiti \, |\cf \kti(2^\ell \xi)|^2 \cn_{\bh}(\xi) \cn_{\bh}(\xiti) 
\\
\times  
\prod_{j=1}^{i-1}\left(\mathcal{T}^{(j)}(\xiti_j)\right)^2 
\left(|\xi_i|^2 \mathcal{Q}^{(i)}(\xi_i,\xiti_i)^2\right)
\prod_{j=i+1}^d\left(\mathcal{T}^{(j)}(\xi_{j}+\xiti_{j})\right)^2\, .
\end{multline*}
By applying Lemma \ref{lem:estim-fourier-kti}, we can assert that for all $a_1,\ldots,a_d\in [0,1]$ such that $a_1+\ldots+a_d <1$,
\begin{align*}
\cj^{i,\ell}_{\mathbf{s}}&\lesssim 2^{-4\ell(a_1+\ldots+a_d)}  \prod_{r=1}^{i-1} \bigg( \int_{\R} d\xiti_r \, \frac{\mathcal{T}^{(r)}(\xiti_r)^2}{|\xiti_r|^{2H_r-1}} \bigg)\times\prod_{k=1}^{i-1}\bigg( \int_{D_{\mathbf{s}_k}} \frac{d\xi_k}{|\xi_k|^{4a_k+2H_k-1}} \bigg)\\
&\ \ \times \bigg(\int_{D_{\mathbf{s}_i}\times \R} d\xi_i d\xiti_i \, \frac{\mathcal{Q}^{(i)}(\xi_i,\xiti_i)^2}{|\xi_i|^{(4a_i+2H_i-2)-1}|\xiti_i|^{2H_i-1}} \bigg)\prod_{p=i+1}^d \bigg( \int_{D_{\mathbf{s}_p}\times \R} d\xi_p d\xiti_p \, \frac{\mathcal{T}^{(p)}(\xi_{p}+\xiti_{p})^2}{|\xi_p|^{4a_p+2H_p-1} |\xiti_p|^{2H_p-1}} \bigg) \, .
\end{align*}
Based on  the criteria of Lemma \ref{lem:t-i-q-i-improv} (which clearly remain true for $\mathcal{T}^{(i)}$ and $\mathcal{Q}^{(i)}$ defined by \eqref{defi-t-i-spa}-\eqref{defi-q-i-spa}), we deduce the following conditions on $a_1,\ldots,a_d$ (to ensure finiteness of the above integrals):
\begin{equation}\label{cond-a-1-2-spa}
\left\{
\begin{array}{ll}
2H_k<4a_k+2H_k <2 & \text{for} \ k\in \{k\in \{1,\ldots,i-1\}: \, \mathbf{s}_k=-\}\\
2<4a_k+2H_k <4+2H_k& \text{for} \ k\in \{k\in \{1,\ldots,i-1\}: \, \mathbf{s}_k=+\} \\
\max(2H_i,3-2H_i)<4a_i+2H_i <4\\
2H_p<4a_p+2H_p <2 & \text{for} \ p\in \{p\in \{i+1,\ldots,d\}: \, \mathbf{s}_p=-\}\\
3-2H_p<4a_p+2H_p <4+2H_p& \text{for} \ p\in \{p\in \{i+1,\ldots,d\}: \, \mathbf{s}_p=+\} .
\end{array}
\right.
\end{equation}
With \eqref{tech-bou-1-2-spa} in mind, we need these inequalities to be also consistent with the relation $\sum_{k=1}^d a_k=1-\varepsilon$. The combination of these two constraints thus leads us to the condition
\begin{align}\label{condit-a-i-spa}
A^i_{\mathbf{s}}< 4(1-\varepsilon)+2H<B^i_{\mathbf{s}}\, ,
\end{align}
with two parameters $A^i_{\mathbf{s}}, B^i_{\mathbf{s}}$ defined by
\begin{align}
A^i_{\mathbf{s}}&:=2\sum_{\substack{k=1,\ldots,i-1\\\mathbf{s}_k=-}} H_k+2\big|\{k\in \{1,\ldots,i-1\}: \, \mathbf{s}_k=+\}\big|+\max(2H_i,3-2H_i)\nonumber\\
&\hspace{2cm}+2\sum_{\substack{p=i+1,\ldots,d\\ \mathbf{s}_p=-}} H_p+\sum_{\substack{p=i+1,\ldots,d\\\mathbf{s}_p=+}} (3-2H_p)\, ,\label{defi-a-+-i-spa}
\end{align}
\begin{align*}
B^i_{\mathbf{s}}&:=2\big|\{k\in \{1,\ldots,i-1\}: \, \mathbf{s}_k=-\}\big|+\sum_{\substack{k=1,\ldots,i-1\\\mathbf{s}_k=+}} (4+2H_k)\\
&\hspace{3cm}+4+2\big|\{p\in \{i+1,\ldots,d\}: \, \mathbf{s}_p=-\}\big|+\sum_{\substack{p=i+1,\ldots,d\\\mathbf{s}_p=+}} (4+2H_p)\, .
\end{align*}
Before checking \eqref{condit-a-i-spa}, observe that due to condition \eqref{strengthened 4.17-spa}, it holds that $d-\frac32<H<H_1+H_2+(d-2)$ (recall that $d\geq 2$), and so
\begin{equation}\label{sum-h-1-h-2}
H_1+H_2>\frac12 \, .
\end{equation}
Besides, for symmetry reasons, we can assume (from the beginning) that $H_1\leq H_2 \leq \ldots \leq H_d$, and consequently, for $d\geq 3$ and $i\geq 3$, $d-\frac32<H<H_1+H_2+H_3+(d-3)\leq 3H_i+(d-3)$, so that
\begin{equation}\label{observ-h-i-spa}
H_i >\frac12 \quad \text{for any} \ i \geq 3 \, .
\end{equation}
Let us now back to the verification of \eqref{condit-a-i-spa}. In order to see that $A^i_{\mathbf{s}}<4+2H$, observe first that 
\begin{equation}\label{bound-a-i-s-spa}
A^i_{\mathbf{s}}<2(i-1)+\sum_{q=i}^d\max(2H_q,3-2H_q) \, .
\end{equation}
By \eqref{observ-h-i-spa}, we immmediately deduce that for $i\geq 3$, $A^i_{\mathbf{s}}<2(i-1)+2(d-i+1)=2d<4+2H$, where the last inequality stems from the assumption $d-\frac32<H$. Then, using again \eqref{bound-a-i-s-spa} and \eqref{observ-h-i-spa}, 
\begin{align*}
&A^2_{\mathbf{s}}<2d-2+\max(2H_2,3-2H_2)<2d+1<4+2H \, .
\end{align*}
Finally, for $A^1_{\mathbf{s}}$, we get by \eqref{bound-a-i-s-spa} and \eqref{observ-h-i-spa} that
\begin{align*}
A^1_{\mathbf{s}}&<\max(2H_1,3-2H_1)+\max(2H_2,3-2H_2)+2(d-2)\\
&<2d-4+\max(2,3-2H_1)+\max(2,3-2H_2)\\
&\leq 2d-4+\max(4,5-2H_1,5-2H_2,6-2(H_1+H_2))<2d+1<4+2H \, ,
\end{align*}
where we have used \eqref{sum-h-1-h-2} to get the fourth inequality.

\smallskip

For the second inequality in \eqref{condit-a-i-spa}, let us write $B^i_{\mathbf{s}}$ as
\begin{align*}
B^i_{\mathbf{s}}&=2(i-1)+2\sum_{\substack{k=1,\ldots,i-1\\\mathbf{s}_k=+}} (1+H_k)+4+2(d-i)+2\sum_{\substack{p=i+1,\ldots,d\\\mathbf{s}_p=+}} (1+H_p)\\
&=2d+2+2\sum_{\substack{k=1,\ldots,i-1\\\mathbf{s}_k=+}} (1+H_k)+2\sum_{\substack{p=i+1,\ldots,d\\\mathbf{s}_p=+}} (1+H_p)\, ,
\end{align*}
and now it becomes clear that $B^i_{\mathbf{s}}>2+2d\geq 4+2H$, since $H\leq d-1$.

\smallskip

This completes the proof of \eqref{condit-a-i-spa}, and accordingly the proof of \eqref{tech-bou-1-2-spa} and \eqref{eq:hatW^n-hatW^m-spa}.

\end{proof}

\subsection{Conclusion: proof of Theorem \ref{theo:conv-k-rp-spa}.}

With Propositions \ref{prop:estim-mom-first-spa} and \ref{prop:estim-mom-second-spa} in hand, we are exactly in the same position as in Section \ref{subsec:conclusion}, and accordingly we can reproduce the exact same reasoning in order to conclude.

\section{Appendix}

\subsection{Proof of Lemma \ref{lem:cstt-renorm}}\label{subsec:proof-renorm-cstt-tech}

We only focus on the treatment of $\cj_{\rho,H_0,\bh}$ (defined in \eqref{eq:cj-c-n}) when $2H_0+H<d+1$. It should however be clear to the reader that the subsequent arguments could also be used to prove the finiteness of the integral in \eqref{eq:cj-c-n-border} when $2H_0+H=d+1$.

\smallskip

According to the definition \eqref{eq:heat-kernel} of the heat kernel $p$ and recalling that $\mathcal{F}$ stands for the space-time Fourier transform, it is readily checked that for $(\la,\xi)\in\R^{d+1}$ we have
\begin{equation}\label{eq:fourier-p}
\mathcal F{p}(\la,\xi)=\lp \frac{|\xi|^2}{2}+\imath \la\rp^{-1}.
\end{equation}
Therefore, the integral under consideration can be bounded as 
\begin{align}\label{bound:J_rho H}
\cj_{\rho,H_0,\bh} \leq \cj_{\infty} +\cj_{0},
\end{align}
where we consider a compact region $\cd_\scal$ of $\R^{d+1}$ defined by
\begin{align}\label{def: calD_s}\cd_\scal:=\{(\la,\xi)\in \R^{d+1}: \, \la^2+\xi_1^4+\cdots+\xi_d^4\leq 1\},\end{align}
and where the quantities $\cj_\infty, \cj_0$ are respectively defined by
\begin{eqnarray}\label{def:J_0}
\cj_{\infty}
&:=&
\int_{\R^{d+1}\backslash \cd_\scal}  \frac{ d\la d\xi}{(\la^2+\xi_1^4 +\cdots+\xi_d^4)^{1/2}}|\mathcal F{\rho}(\la,\xi)|^2 \cn_{H_0,\bh}(\la,\xi)
\notag\\
\cj_{0}
&:=&
\int_{\cd_\scal}  \frac{ d\la d\xi}{(\la^2+\xi_1^4 +\cdots+\xi_d^4)^{1/2}}|\mathcal F{\rho}(\la,\xi)|^2 \cn_{H_0,\bh}(\la,\xi) \, .
\end{eqnarray}
We now proceed to the evaluation of those two terms.

In order to estimate $\cj_\infty$, note that $(\R^{d+1}\backslash \cd_\scal) \subset \cup_{i=0}^d \Lambda_i$, where the regions $\Lambda_i$ are defined by
$$\Lambda_0:=\left\{(\la,\xi_1,\ldots,\xi_d): \, \la^2\geq \frac{1}{d+1}\right\} \quad\text{and} \quad \Lambda_i:=\left\{(\la,\xi_1,\ldots,\xi_d): \, \xi_i^4\geq \frac{1}{d+1}\right\} \, .$$
According to this decomposition we write 
\begin{align}\label{def: decompose J_infty}
\cj_{\infty}\leq\sum_{i=0}^d\cj_{\infty,i},
\end{align}
where the terms $\cj_{\infty,i}$ can be written as
\begin{align}\label{def:J_infty i}
\cj_{\infty,i}:=\int_{\Lambda_i}  \frac{ d\la d\xi}{(\la^2+\xi_1^4 +\cdots+\xi_d^4)^{1/2}}|\mathcal F{\rho}(\la,\xi)|^2 \cn_{H_0,\bh}(\la,\xi) \, .
\end{align}
Let us now show how to bound $\cj_{\infty,0}$ above. To this aim we invoke our bound \eqref{asympt-rho} in two different ways.  Namely we take $\tau_0=1$, and $\tau_i=0$ if $|\xi_i|\leq 1$, while $\tau_i=1$ if $|\xi_i|\geq 1$. Together with the trivial inequality $\la^2+\sum_{i=1}^d\xi_i^4\geq \la^2$, the term $\cj_{\infty,0}$ given in \eqref{def:J_infty i} can be bounded as follows
\begin{align}\label{bound:cj_infty0}
\cj_{\infty,0}\lesssim \bigg(\int_{\la^2\geq \frac{1}{d+1}} \frac{ d\la }{|\la|^{2H_0+2}}\bigg) \prod_{i=1}^d \bigg\{ \int_{|\xi_i|\leq 1} \frac{d\xi_i}{|\xi_i|^{2H_i-1}}+\int_{|\xi_i|\geq 1} \frac{d\xi_i}{|\xi_i|^{2H_i+1}} \bigg\} \ < \ \infty \, ,
\end{align}
where the last inequality is immediate. The terms $\cj_{\infty,i}$ for $i=1,\dots,d$ in \eqref{def:J_infty i} are handled similarly, and we omit the details for the sake of conciseness. Taking into account the upper bound~\eqref{def: decompose J_infty}, we end up with the relation $\cj_\infty<\infty$.

We now turn to a bound on $\cj_0$ defined by \eqref{def:J_0}, for which we invoke \eqref{asympt-rho} with $\tau_i=0,$ for all $ i=0,\dots,d$. We get
\begin{equation}\label{cj-g}
\cj_0\lesssim\int_{\cd_\scal \cap \R_+^{d+1}}  \frac{ d\la d\xi}{(\la^2+\xi_1^4 +\cdots+\xi_d^4)^{1/2}}\cn_{H_0,\bh}(\la,\xi)\, .
\end{equation}
To see that the latter integral is indeed finite, let us set $\tilde{\xi}_i:=\xi_i^2$, so that $(\la,\xi_1,\ldots,\xi_d)\in \cd_\scal \cap \R_+^{d+1}$ if and only if $(\la,\tilde{\xi}_1,\ldots,\tilde{\xi}_d)\in \cb(0,1) \cap \R_+^{d+1}$, where $\cb(0,1)$ stands for the standard Euclidean unit ball. This yields   
\begin{equation}\label{cj-g-bis}
\cj_0 \lesssim \int_{\cb(0,1)\cap \R_+^{d+1}}  \frac{d\la d\xi}{(\la^2+\tilde{\xi}_1^2 +\cdots+\tilde{\xi}_d^2)^{1/2}} \frac{1}{|\la|^{2H_0-1}} \prod_{i=1}^d \frac{1}{|\tilde{\xi}_i|^{H_i}}\lesssim \int_0^1 \frac{dr}{r^{2H_0+H-d}} \, , 
\end{equation}
where we have used spherical coordinates to derive the last inequality. The finiteness of $\cj_0$ now follows from the assumption $2H_0+H<d+1$.

Summarizing our computations, we have seen that $\cj_0<\infty$ and $\cj_\infty<\infty$. Recalling relation \eqref{bound:J_rho H}, this proves our claim $\cj_{\rho,H_0,\bh}<\infty$.

\subsection{Proof of Proposition \ref{prop:asymp-renorm-cstt}}\label{subsec:proof-asymp}

Let us decompose the integral under consideration as
\begin{align}
&\int_{|\la|+|\xi|^2\geq 2^{-2n}} |\mathcal F{\rho}(\la,\xi)|^2  \mathcal F{p}(\la,\xi) \cn_{H_0,\bh}(\la,\xi)\, d\la d\xi=\int_{2^{-2n}\leq |\la|+|\xi|^2\leq 1}   \mathcal F{p}(\la,\xi) \cn_{H_0,\bh}(\la,\xi)\, d\la d\xi\nonumber\\
&\hspace{3.5cm}+\int_{2^{-2n}\leq |\la|+|\xi|^2\leq 1}  \big\{|\mathcal F{\rho}(\la,\xi)|^2-1\big\}  \mathcal F{p}(\la,\xi) \cn_{H_0,\bh}(\la,\xi)\, d\la d\xi+O(1).\label{proof-cstt-bord-1}
\end{align}

Using a series of elementary changes of variable, we get, for some constant $C_{H_0,\bh}$ that may change from line to line,
\begin{align*}
&\int_{2^{-2n}\leq |\la|+|\xi|^2\leq 1}   \mathcal F{p}(\la,\xi) \cn_{H_0,\bh}(\la,\xi)\, d\la d\xi=\int_{2^{-2n}\leq |\la|+|\xi|^2\leq 1} \frac{d\la d\xi}{\frac{|\xi|^2}{2}+\imath \la} \frac{1}{|\la|^{2H_0-1}} \prod_{i=1}^d \frac{1}{|\xi_i|^{2H_i-1}}\\
&\hspace{3cm}=C_{H_0,\bh} \int_0^\infty dr \int_{2^{-2n}\leq |\la|+r^2\leq 1}\frac{d\la }{\frac{r^2}{2}+\imath \la} \frac{r^{2d-2H-1}}{|\la|^{2H_0-1}}\\
&\hspace{3cm}=C_{H_0,\bh} \int_0^\infty dr \int_0^\infty d\la\,  \1_{2^{-2n}\leq \la+r^2\leq 1}\bigg[\frac{1}{\frac{r^2}{2}+\imath \la} +\frac{1}{\frac{r^2}{2}-\imath \la} \bigg] \frac{r^{2d-2H-1}}{|\la|^{2H_0-1}}\\
&\hspace{3cm}=C_{H_0,\bh} \int_0^\infty dr \int_0^\infty d\la\,  \1_{2^{-2n}\leq \la+r^2\leq 1}\bigg(\frac{r^2}{\frac{r^4}{4}+ \la^2}\bigg) \frac{r^{2d-2H-1}}{|\la|^{2H_0-1}}\\
&\hspace{3cm}=C_{H_0,\bh} \int_0^\infty dr \int_0^\infty d\la\,  \1_{2^{-2n}\leq \la^2+r^2\leq 1}\frac{\la}{\frac{r^4}{4}+ \la^4} \frac{r^{2d-2H+1}}{|\la|^{4H_0-2}}\\
&\hspace{3cm}=C_{H_0,\bh} \bigg(\int_0^\infty d\rho  \frac{\1_{2^{-2n}\leq \rho^2\leq 1}}{\rho^{2(2H_0+H)-2d-1}}\bigg) \bigg(\int_0^{\frac{\pi}{2}} \, \frac{d\theta }{\frac{\cos^4 \theta}{4}+ \sin^4 \theta} \frac{(\cos \theta)^{2d-2H+1}}{(\sin \theta)^{4H_0-3}}\bigg)
\end{align*}
and so, recalling that $2H_0+H=d+1$, we end up with
\begin{equation}\label{proof-cstt-bord-2}
\int_{2^{-2n}\leq |\la|+|\xi|^2\leq 1}   \mathcal F{p}(\la,\xi) \cn_{H_0,\bh}(\la,\xi)\, d\la d\xi=C_{H_0,\bh} \bigg(\int_{2^{-n}}^1 \frac{d\rho}{\rho}  \bigg)=C_{H_0,\bh} \cdot n \, . 
\end{equation}

\

On the other hand, thanks to Assumption $(\rho)$-$(i)$-$(ii)$, we have
\begin{align*}
&\int_{2^{-2n}\leq |\la|+|\xi|^2\leq 1}  \big| |\mathcal F{\rho}(\la,\xi)|^2-1\big|  \big| \mathcal F{p}(\la,\xi)\big| \cn_{H_0,\bh}(\la,\xi)\, d\la d\xi\\
&=\int_{2^{-2n}\leq |\la|+|\xi|^2\leq 1}  \big| |\mathcal F{\rho}(\la,\xi)|^2-|\mathcal F{\rho}(0,0)|^2\big|  \big| \mathcal F{p}(\la,\xi)\big| \cn_{H_0,\bh}(\la,\xi)\, d\la d\xi\\
&\lesssim \int_{0\leq |\la|+|\xi|^2\leq 1}  \big\{ |\la|+|\xi|\big\} \big| \mathcal F{p}(\la,\xi)\big| \cn_{H_0,\bh}(\la,\xi)\, d\la d\xi\\
&\lesssim \int_0^\infty dr \int_0^\infty d\la \, \1_{0\leq \la+r^2\leq 1}\,   \big\{ \la+r\big\} \frac{r^{2d-2H-1}}{r^2+\la} \frac{1}{\la^{2H_0-1}}\\
&\lesssim \int_0^\infty dr \int_0^\infty d\la \, \1_{0\leq \la^2+r^2\leq 1}\,  \la \big\{ \la^2+r\big\} \frac{r^{2d-2H-1}}{r^2+\la^2} \frac{1}{\la^{4H_0-2}}\\
&\lesssim \int_{0\leq \rho^2\leq 1} d\rho \, \rho^3 \frac{\rho^{2d-2H-1}}{\rho^2} \frac{1}{\rho^{4H_0-2}} \lesssim \int_{0\leq \rho^2\leq 1} \frac{d\rho}{\rho^{2(2H_0+H)-2d-2}} \lesssim 1 \, , 
\end{align*}
where the last inequality is immediately derived from the assumption $2H_0+H=d+1$. Thus,
\begin{equation}\label{proof-cstt-bord-3}
\sup_{n\geq 1} \bigg|\int_{2^{-2n}\leq |\la|+|\xi|^2\leq 1}  \big\{|\mathcal F{\rho}(\la,\xi)|^2-1\big\}  \mathcal F{p}(\la,\xi) \cn_{H_0,\bh}(\la,\xi)\, d\la d\xi\bigg| \ < \ \infty \, .
\end{equation}

\smallskip

Finally, injecting \eqref{proof-cstt-bord-2} and \eqref{proof-cstt-bord-3} into \eqref{proof-cstt-bord-1}, we deduce the desired decomposition \eqref{decompo-cstt-border}.


\bigskip


\begin{thebibliography}{1}

\bibitem{BCR09}
R. Bass, X. Chen and J. Rosen: Large deviations for Riesz potential of additive processes. {\it Ann. Inst. Henri Poincar\'{e} Probab. Stat.} {\bf 45} (2009), no. 3, 626-666.



\bibitem{CCGT}
P. Chakraborty,  X. Chen, B. Go and S. Tindel:
Quenched asymptotics for a 1-d stochastic heat equation driven by a rough spatial noise.
To appear in {\it Stoch. Proc. Appl.}






\bibitem{Ch-bk}
X. Chen: \emph{Random Walk Intersections: Large Deviations and Related Topics.}
American Mathematical Society. (2008)



\bibitem{Ch12}
X. Chen: Quenched asymptotics for Brownian motion of renormalized Poisson potential and for the related parabolic Anderson models.
{\it Ann. Probab.} {\bf 40} (2012), no. 4, 1436-1482.



\bibitem{Ch14}
X. Chen: Quenched asymptotics for Brownian motion in generalized Gaussian potential.
{\it Ann. Probab.} {\bf 42} (2014), no. 2, 576-622.



\bibitem{Ch17}
X. Chen: Moment asymptotics for parabolic Anderson equation with fractional time-space noise in Skorokhod regime. 
{\it Ann. Inst. Henri Poincar\'e Probab. Stat.} {\bf 53} (2017), no. 2, 819-841.



\bibitem{Ch18}
X. Chen: Parabolic Anderson model with rough or critical Gaussian noise.
{\it Ann. Institut Henri Poincar\'e Probab. Stat.} {\bf 55} (2019), no. 2, 941-976.  



\bibitem{Ch19}
X. Chen: Parabolic Anderson model with a fractional Gaussian noise that is rough in time. 
{\it Ann. Institut Henri Poincar\'e Probab. Stat.} {\bf 56} (2020), no. 2, 792-825.



\bibitem{CDOT}
X. Chen, A. Deya, C. Ouyang and S. Tindel: Moment estimates for some renormalized parabolic Anderson models. {\it Submitted} (2020).



\bibitem{CHNT}
X. Chen, Y. Hu, D. Nualart and S. Tindel:
Spatial asymptotics for the parabolic Anderson model driven by a Gaussian rough noise. 
{\it Electron. J. Probab.} {\bf 22} (2017).



\bibitem{CHSX} 
X. Chen, Y. Hu, J. Song and F. Xing:
Exponential asymptotics for time-space Hamiltonians.
{\it Ann. Institut Henri Poincar\'e Probab. Stat.} {\bf 51} (2015), 1529-1561.



\bibitem{CFK}
D. Conus and D. Khoshnevisan: 
On the existence and position of the farthest peaks of a family of stochastic heat and wave equations. 
{\it Probab. Theory Related Fields} {\bf 152} (2012), no. 3-4, 681-701.



\bibitem{CFJK}
D. Conus, M. Foondun, M. Joseph and D. Khoshnevisan:
On the chaotic character of the stochastic heat equation II.
\textit{Probab. Theory Related Fields} {\bf 156} (2013), no. 3-4, 483-533.



\bibitem{CJK}
D. Conus, M. Joseph, and D. Khoshnevisan: 
On the chaotic character of the stochastic heat equation, before the onset of intermittency.
\textit{Ann. Probab.} \textbf{41} (2013), no. 3B, 2225-2260.



\bibitem{De16}
A. Deya: 
On a modelled rough heat equation.
{\it Probab. Theory Relat. Fields} {\bf 166} (2016), 1-65.



\bibitem{De17}
A. Deya: 
Construction and Shorohod representation of a fractional $K$-rough path. 
{\it Electron. J. Probab.} {\bf 22} (2017).



\bibitem{DZ}
A. Dembo and O. Zeitouni: 
{\it Large deviations techniques and applications}. Second edition. Applications of Mathematics (New York), {\bf 38}. Springer-Verlag, 1998. 



\bibitem{GH}
Y. Gu and W. Xu: 
Moments of 2D parabolic Anderson model. 
{\it Asymptot. Anal.} {\bf 108} (2018), no. 3, 151-161.



\bibitem{hai-14}
M. Hairer: A theory of regularity structures. {\it Invent. Math.} {\bf 198} (2014), no. 2, 269-504.



\bibitem{HL}
M. Hairer and C. Labb{\'e}: 
Multiplicative stochastic heat equations on the whole space. 
{\it J. Eur. Math. Soc.}  {\bf 20} (2018), no. 4, 1005-1054.



\bibitem{HKP}
T. Hida, H-H. Kuo, J. Potthoff and L. Streit:
{\it White noise. An infinite-dimensional calculus.} 
Kluwer Academic Publishers, 1993.



\bibitem{HHNT}
Y. Hu, J. Huang, D. Nualart and S. Tindel: Stochastic Heat Equations with General Multiplicative Gaussian Noises: H\"older Continuity and Intermittency.
{\it Electron. J. Probab.} {\bf 20} (2015), no. 55, 1-50.



\bibitem{HN} 
Y. Hu and D. Nualart: 
Stochastic heat equation driven by fractional noise and local time.  
\textit{Probab. Theory Related Fields} \textbf{143} (2009), no. 1-2, 285-328.

\bibitem{HLN}
J. Huang, K. L\^e, D. Nualart:
Large time asymptotics for the parabolic Anderson model driven by space and time correlated noise. 
{\it Stoch. Partial Differ. Equ. Anal. Comput.} {\bf 5} (2017), no. 4, 614-651.



\bibitem{Kh}
D. Khoshnevisan: Analysis of stochastic partial differential equations. 
CBMS Regional Conference Series in Mathematics, 119. American Mathematical Society, 2014.



\bibitem{konig_book}
W. K\"onig: \emph{The Parabolic Anderson Model: Random Walk in Random Potential.}
Birkh\"auser (2016).



\bibitem{Le}
K. L\^e:
A remark on a result of Xia Chen. 
{\it Statistics \& Probability Letters} {\bf 118} (2016), 124-126.



\bibitem{Nu-bk} 
D. Nualart : {\it The Malliavin Calculus and Related Topics.} 
Second edition. Probability and its Applications (New York). Springer-Verlag, Berlin, 2006.



\bibitem{Nu-Za}
D. Nualart and  M. Zakai: Generalized multiple stochastic integrals and the representation of Wiener functionals. 
\textit{Stochastics} \textbf{23} (1988), 311-330.





\end{thebibliography}
\end{document}